\documentclass[12pt]{amsart}
\usepackage[utf8]{inputenc}
\usepackage{amsmath}
\usepackage{bbm}
\usepackage{hyperref}

\usepackage{latexsym}
\usepackage{xcolor}
\usepackage{enumerate}
\usepackage{epsfig}
\usepackage{epstopdf}

\usepackage{subcaption}

\usepackage[usenames,dvipsnames]{pstricks}
\usepackage{pst-grad} 
\usepackage{pst-plot} 

\newtheorem{thm}{Theorem}
\newtheorem{lemma}[thm]{Lemma}
\newtheorem{cor}[thm]{Corollary}

\theoremstyle{remark}

\newcommand{\beq}{\begin{equation}}
\newcommand{\eeq}{\end{equation}}

\usepackage{xcolor}

\newcommand{\CC}{\mathbb{C}}

\newcommand{\re}{\mathrm{Re}}
\newcommand{\im}{\mathrm{Im}}

\title[Certified counting of roots of random univariate polynomials]{Certified counting of roots of random univariate polynomials}

\author[\mbox{ } \mbox{ } Cleveland, Dzugan, Hauenstein, Haywood, Mehta, Morse, Robol, Schlenk]{Joseph Cleveland, 
Jeffrey Dzugan, Jonathan D. Hauenstein, \\
Ian Haywood, Dhagash Mehta, Anthony Morse, \\ Leonardo Robol and Taylor Schlenk}

\address{Joseph Cleveland \\
Sam Houston State University
}
\email{jac087@shsu.edu}

\address{Jeffrey Dzugan \\
         Department of Mathematics\\
         Samford University}
\email{jdzugan@samford.edu }

\address{Jonathan D. Hauenstein \\
         Department of Applied and Computational Mathematics and Statistics\\
 University of Notre Dame\\
 Notre Dame, IN 46556\\
         USA}
\email{hauenstein@nd.edu}
\urladdr{\url{www.nd.edu/\~jhauenst}}

\address{Ian Haywood
Department of Mathematics\\
         North Carolina State University\\
         Raleigh\\
         North Carolina \ 27695\\
         USA
}
\email{ithaywoo@ncsu.edu}

\address{Dhagash Mehta \\
         Department of Applied and Computational Mathematics and Statistics\\
 University of Notre Dame\\
 Notre Dame, IN 46556\\
         USA}
\email{dmehta@nd.edu}
\urladdr{\url{www.nd.edu/\~dmehta}}

\address{Anthony Morse
Department of Mathematics\\
         State University of New York, Brockport\\
         USA
}
\email{anthonymorse92@gmail.com}

\address{Leonardo Robol
Scuola Normale Superiore, \\
Piazza dei Cavalieri, 7 - 56127 Pisa\\
Italy.
}
\email{leonardo.robol@sns.it}

\address{Taylor Schlenk
Department of Mathematics\\
         University of North Dakota
}
\email{taymarsch@hotmail.com}

\thanks{Research was supported by NSF grants DMS-1063010 and
DMS-1262428, NSA project number H98230-12-1-0299,
DARPA YFA, and Sloan Research Fellowship. 
The authors would like to thank Kambiz Farahmand, Antonio Lerario, 
Erik Lundberg and Yan Fyodorov for their feedback.}

\begin{document}
\begin{abstract}

A challenging problem in computational mathematics is to
compute roots of a high-degree univariate random polynomial.
We combine an efficient multiprecision implementation 
for solving high-degree random polynomials with two 
certification methods, namely Smale's $\alpha$-theory and one 
based on Gerschgorin's theorem, for showing 
that a given numerical approximation is in the quadratic 
convergence region of Newton's method of some exact solution. 
With this combination, we can {\it certifiably count}
the number of real roots of random polynomials.  We quantify the difference
between the two certification procedures and list the salient features of both of them. 
After benchmarking on random polynomials where the coefficients are drawn from the Gaussian distribution, 
we obtain novel experimental results for the~Cauchy~distribution~case.
\end{abstract}

\maketitle

\section{Introduction}
A random univariate polynomial is a univariate polynomial whose coefficients are picked from a random distribution.
Studying the roots of random polynomials is a classic problem in mathematics due to its numerous applications in many areas of mathematics 
and engineering, e.g., random matrix theory, hydrology, meteorology, aerodynamics, and structural engineering~\cite{Culling:88,Farahmand:98}. 
Other situations also lead to solving 
random polynomials including solving 
polynomials where coefficients are subject to random noise; 
the solutions of differential equations with random coefficients; 
problems in spectral theory of random matrices;
polynomial regression;
and the development computer algebra methods~\cite{Bharucha-Reid:86,Pan:97}.

Recently, more applications have emerged from theoretical physics, especially in statistical physics, 
e.g.,~\cite{fyodorov2012freezing,schehr2007statistics},
quantum mechanics, e.g.,~\cite{Bogomolny:92}, 
and string theory, e.g.,~\cite{Ashmore:2011yw,He:2010sf}. 
The roots of certain random polynomials are also related to minimizing the 
energy of charged particles on a sphere~\cite{Armentano:11}, i.e., Smale's 7th problem~\cite{Smale:98}.

Formally, a random polynomial is given as
\begin{equation}
 f(x) = \sum_{i=0}^{N} a_i x^i,\label{eq:univariate_random}
\end{equation}
where $N$ is the degree of the polynomial and the coefficients $a_{i}$ 
are random numbers picked from a distribution. 
The Fundamental Theorem of Algebra yields that the number of complex roots of $f$ is $N$, counting multiplicity.
Since writing the exact roots of $f$ as a general algebraic expression of the coefficients $a_i$
for $N\geq 5$ is not possible in general, due to the Abel-Ruffini theorem, we use 
numerical root-finding methods. 
There are numerous numerical methods available to find roots of univariate polynomials with floating point coefficients, e.g., 
companion matrices, the Aberth method, and the Jenkins-Traub method.  
All of these methods work well in certain cases
(see Ref.~\cite{McNamee:13} for an extensive review).
There are also several symbolic methods to solve polynomials 
with rational coefficients, but the random polynomials under consideration 
will generally not have rational coefficients.

Approximating the roots of univariate polynomials with moderate degrees starts becoming an increasingly difficult task.  However, in theory,
numerical approximation methods typically have complexity around $O(N^2)$,
but have the disadvantage of yielding only approximated results. 
This introduces the need to define what we actually consider
a ``good approximation'' of a polynomial root. ``Good'' numerical methods
often are backward-stable, that is, the results obtained at the end are
exact results of a slightly perturbed input. Unfortunately, for ill-conditioned
problems, even if they are posed in an exact setting, 
this might not be enough to provide bounds on the quality of the result. 

Smale defined ``good enough'' as being in the quadratic convergence zone
of Newton's method for an exact root of the polynomial 
so that the exact root can be efficiently 
approximated to arbitrary accuracy.  
A point in such a quadratic convergence zone
is called an {\it approximate root} and Smale showed how to certify using
only data computed at the test point that it was an indeed an approximate root
via $\alpha$-theory.
An implementation of $\alpha$-theory is provided in 
{\tt alphaCertified}~\cite{2010arXiv1011.1091H,alphaCertified}.
Other results of this kind also exist in the literature, e.g., \cite{renegar,tilli}.

Other certification techniques are available for the univariate case,
which, as we will see, can lead to better inclusion results. 
In this paper, we compare $\alpha$-theory with a
global approach used 
in the MPSolve package \cite{mpsolve,secsolve}. 
The MPSolve package implements an efficient multiprecision univariate polynomial solver
which our tests show is efficient and effective at solving 
polynomials even as the degree increases.

The main inclusion result, originally proved in the paper by Tilli \cite{tilli}, 
roughly states that if the distance of an approximation from the root is less 
than $s/N$, where $s$ is a constant depending on the separation of the roots, 
then Newton's method converges quadratically from the start. 
The constant $s$ can be easily computed or bounded 
without computing the roots by means
of Gerschgorin's theorem applied to a particular companion matrix. 
One may view this as a global method since $s$ depends on all of the roots
whereas $\alpha$-theory is a local method.  
Since we are focusing on random polynomials where, with probability $1$,
$f$ has $N$ distinct roots, the global root separation bound will be positive.
Thus, this global method is applicable for the problems under consideration.

The output of the software MPSolve is well-suited for this random context: 
it can deliver good approximations using multiprecision arithmetic
(given exact input, it obtains the roots to an arbitrary number of digits), can guarantee the inclusion radii, and can certify if a root is real. 
In this way, we were able to solve large samples of random polynomials and to certifiably count the number of real roots. 

For concrete problems, we work with two different sets of random polynomials: 
coefficients are independently picked from the Gaussian distribution with mean zero and variance one, 
and from the standard Cauchy distribution. 
Since the Gaussian case is well studied with many analytical results available,
we use this as our benchmark system. 
The Cauchy case is relatively under-studied and, except for the mean number of real roots, 
analytical results are usually not available. 
Hence, our numerical results for the Cauchy case can provide much needed input 
for potentially leading to new analytical results. 

The rest of the paper is organized as follows. 
Section~\ref{sec:Numerical_Methods} describes numerical methods under consideration for approximating
roots of univariate polynomials as well as discusses two certification methods.
We also pictorially show the certifiable regions of a simple and exactly solvable polynomial equation. 
Section~\ref{sec:numerical_set_up} describes our setup for solving and certifying the roots
of random polynomials.  Section~\ref{sec:results} highlights 
some differences between the two certification methods, 
especially as the degree increases,
and provides numerical results for various quantities associated
to random polynomials.

\section{Numerical Methods}\label{sec:Numerical_Methods}

One common approach for solving univariate polynomial equations is to use 
the companion matrix.  
This method is widely used in mathematical software such as {\sc Matlab}. 
After describing the companion matrix method, we summarize
Smale's $\alpha$-theory followed by the certification approach
based on Gerschgorin's theorem used in MPSolve.  
The global nature of certification based on Gerschgorin's theorem 
generally provides larger certification regions due to exploiting more information
about the roots than the local $\alpha$-theory approach.
This section ends with certification methods for detecting real roots.

\subsection{The Companion Matrix Method}

Suppose that for $f$ presented in \eqref{eq:univariate_random}, 
the leading coefficient $a_N$ is nonzero, i.e., $f$
is a univariate polynomial of degree~$N$.  
One can find the roots of 
$f$ by computing the eigenvalues of the companion matrix:
\begin{equation}
C_f := \left( \begin{array}{ccccc}
 0 & 0 & 0 & 0 & -a_{0}/a_{N} \\
  1 & 0 & 0 & 0 & -a_{1}/a_{N} \\
 0 & 1 & 0 & 0 & -a_{2}/a_{N} \\ 
  0 & 0 & 1 & 0 & . \\
  . & . & . & . & . \\
  . & . & . & . & . \\
  0 & 0 & \dots & 1 & -a_{N-1}/a_{N} \\
 \end{array} \right),
\end{equation}
since $f$ is the characteristic polynomial of $C_f$.
The eigenvalues of $C_f$ can be computed, for example, using
the QR algorithm as implemented in LAPACK \cite{anderson1999lapack}. 
The {\tt roots()} command in {\sc Matlab} uses this companion matrix method.

\subsection{Smale's $\alpha$-Theorem}

Due to conditioning of $C_f$, the approximations of the roots of $f$ 
computed by, for example, LAPACK or {\sc Matlab}, using double precision
may not be in a quadratic convergence zone of Newton's method.  
One way to certify they are indeed in a quadratic convergence zone 
is to use $\alpha$-theory.
Although the approach works for more general situations, we focus
on univariate polynomials $f$ of degree $N$.
Let $f'$ be the first derivative of $f$ and $f^{(k)}$ be
the $k^{\rm th}$ derivative.  
A Newton iteration of $f$ starting at $x^*\in\CC$ is defined by
\begin{equation*}
N_{f}(x^*):=\begin{cases}
x^*-f(x^*)/f'(x^*) & \mbox{if }f'(x^*)\neq0,\\
 x^* & \mbox{otherwise.}
\end{cases}
\end{equation*}
Clearly, if $f'(x^*)\neq0$, then $N_f(x^*) = x^*$ if and only if $f(x^*) = 0$.

Newton's method relies upon repeatedly applying Newton iterations.  
For simplicity, the $k^{\rm th}$ Newton iteration will be denoted
\begin{equation*}
N_{f}^{k}(x):=\underbrace{N_{f}\circ\dots\circ N_{f}}_{\hbox{$k$ times}}(x).
\end{equation*}

After defining the notion of an approximate root, we state a theorem
which can be used to certify that a point is indeed an approximate root
without {\em a priori} knowledge about the exact root.
A point $x^*\in\CC$ is an \emph{approximate root} of $f$ 
with associated root $\zeta$ if, for each $k\geq1$,
\begin{equation}
\left|N_{f}^{k}(x^*)-\zeta\right|\leq\left(\frac{1}{2}\right)^{2^{k}-1}\left|x^*-\zeta\right|.
\end{equation}
That is, $f(\zeta) = 0$ and $x^*$ is in the quadratic convergence basin of
$\zeta$ for $f$.

\begin{thm}[\cite{BCSS}, pg. 160] \label{thm:alpha}
If $x^*\in\CC$ and $f$ is a univariate polynomial of degree~$N$ such that $f'(x^*)\neq0$ and
$$\alpha(f,x^*) < \frac{13-3\sqrt{17}}{4},$$
then $x^*$ is an approximate root of $f$ 
with associated root $\zeta$ where 
\begin{equation}
\begin{array}{cl}
\alpha(f,x):= & \beta(f,x)\cdot \gamma(f,x),\\
\noalign{\smallskip}
\beta(f,x):= & |f(x)/f'(x)|, \quad\ \mbox{and}\\
\noalign{\smallskip}
\gamma(f,x):= & \underset{2\leq k\leq N}{\max}
\left|\frac{\displaystyle f^{(k)}(x)}{\displaystyle k!\cdot f'(x)}\right|^{1/(k-1)}.
\end{array}
\end{equation}
Moreover, $|x^* - \zeta| <2\beta(f,x^*)$.
\end{thm}

Given two approximate roots $x_1$ and $x_2$, they must have distinct associated roots, via
the triangle inequality, if
$$|x_1 - x_2| > 2(\beta(f,x_1) + \beta(f,x_2)).$$
This and several other $\alpha$-theoretic algorithms are implemented in the software package
{\tt alphaCertified} \cite{2010arXiv1011.1091H,alphaCertified} which,
among other things, has been recently used to certify solutions to multivariate 
problems arising in theoretical physics and chemistry \cite{Mehta:2013zia,Mehta:2014gia}.

In the present article, we take advantage of the fact that 
random polynomials of degree $N$ have, with probability $1$, 
$N$ distinct complex roots. 
This means that if we find $N$ approximate roots with distinct associated roots, 
we have accounted for \textit{all} of the roots.

\subsection{Approximate Roots using MPSolve} \label{subsec:mpsolve}

Another way to certify a point is an approximate root of a univariate polynomial
is the following global approach.

\begin{thm}[\cite{tilli}] \label{thm:newtonisolation}
Let $f$ be a monic univariate polynomial of degree $N\geq4$ with 
distinct roots $\zeta_1,\dots,\zeta_N$, i.e., $f(x) = \prod_{i=1}^N (x-\zeta_i)$
with $\zeta_i\neq\zeta_j$ for $i\neq j$. 
Suppose that $x^*\in\CC$ and $1\leq i\leq N$ such that 
$$|x^* - \zeta_i| \leq \frac{|x^* - \zeta_j|}{3(N-1)} \hbox{~~for all~} j \neq i.$$
Then, $x^*$ is an approximate root of $f$ with associated root $\zeta_i$.
\end{thm}

This theorem involves computing the distance from a point to 
the {\em a priori} unknown
roots of $f$.  The following lemma provides an effective bound.

\begin{lemma} \label{lem:newtonradius}
If $f$ is a monic univariate polynomial of degree $N$
and $x^*\in\CC$ such that $f'(x^*)\neq0$, 
then $\{z\in\CC~|~|x^* - z| \leq N\cdot\beta(f,x^*)\}$ contains at least~one~root~of~$f$.
\end{lemma}

\begin{proof}
  See \cite{henrici}.
\end{proof}

The following combines Theorem~\ref{thm:newtonisolation} and Lemma~\ref{lem:newtonradius} 
to allow one to algorithmically verify that a given point is an approximate root
given information about all of the roots.  
This type of global approach is used in the MPSolve package \cite{mpsolve,secsolve}.

\begin{cor}
Let $f$ be a monic univariate polynomial of degree $N\geq4$ with 
$N$ distinct roots and suppose that $x_1,\dots,x_N\in\CC$.
If, for all $i\neq j$,
\begin{equation}\label{eq:distance}
|x_i - x_j| > N(3N-1) (\beta(f,x_i) + \beta(f,x_j))
\end{equation}
then each $x_i$ is an approximate root of $f$ with distinct associated roots.
\end{cor}
\begin{proof}
By Lemma~\ref{lem:newtonradius}, the ball $B_i := B(x_i,N\cdot\beta(f,x_i))$
centered at $x_i$ of radius $N\cdot\beta(f,x_i)$ contains at least one root.
For $i\neq j$, it follows immediately from \eqref{eq:distance}
that $B_i\cap B_j = \emptyset$.
Thus, each $B_i$ contains a distinct root of $f$, say $\zeta_i$.
For $i\neq j$,
$$\begin{array}{lcl}
|x_i - \zeta_j| &\geq& |x_i - x_j| - |x_j - \zeta_j| \\
                & >  & N(3N-2)(\beta(f,x_i) + \beta(f,x_j)) - N\beta(f,x_i) \\
                &\geq& 3N(N-1)(\beta(f,x_i) + \beta(f,x_j)).
\end{array}$$
Thus, for all $i\neq j$,
$$\frac{|x_i - \zeta_j|}{3(N-1)}\geq N(\beta(f,x_i)+\beta(f,x_j))\geq |x_i - \zeta_i|$$
and the statement follows from Theorem~\ref{thm:newtonisolation}.
\end{proof}

The balls $B(x,N\cdot\beta(f,x))$ provides some information about the roots.  
Gerschgorin's theorem is another way to obtain strict and useful inclusion 
results to analyze clusters of roots, such as the following statement.

\begin{thm}
Let $A=(a_{ij})$ be a $N \times N$ matrix with $a_{ij}\in\CC$. 
Let $r_i := \sum_{j \neq i} |a_{ij}|$
and $B_i := B(a_{ii},r_i)$ be the ball centered at $a_{ii}$ with radius $r_i$. 
Then, 
  \begin{enumerate}
    \item each eigenvalue of $A$ is contained in $\bigcup_{i=1}^N B_i$, and
    \item every connected component of this union made up of $k$ balls contains exactly
      $k$ eigenvalues of $A$ (counted with multiplicity).   
  \end{enumerate}
\end{thm}

Suppose that $f$ is a monic univariate polynomial and
$x_1,\dots,x_N\in\CC$ are distinct.
Define $w_i := - f(x_i) \prod_{j\neq i}(x_i - x_j)^{-1}$.
Consider the following matrix:
\[
  A := \left[\begin{array}{cccc} w_1 + x_1 & w_2 & \cdots & w_N \\ w_1 & w_2 + x_2 & \cdots & w_N \\ \vdots & \vdots & \ddots & \vdots \\ w_1 & w_2 & \cdots & w_N + x_N\end{array}\right]
\]

\begin{lemma} The matrix $A$ is a companion matrix for $f$, i.e., 
$\det\left( A - xI \right) = f(x)$. 
\end{lemma}

\begin{proof}
  Let $D = \mathrm{diag}(x_1, \dots, x_N)$, $e = [ 1, \dots, 1 ]^T$ and 
  $w = [ w_1, \dots, w_N ]^T$. Recall that $\det(I + uv^T) = 1 + v^Tu$. 
  We have then that $A = D + ew^T$ so we can write:
  \begin{align*}
    \det\left( xI - A \right) & = \det\left( (xI-D) ( I - (xI-D)^{-1} ew^T ) \right) \\
      & = \prod_{i = 1}^N (x - x_i) \cdot \left( 1 - \sum_{i = 1}^N \frac{w_i}{x - x_i} \right) \\
      & = \prod_{i = 1}^N (x - x_i) + \sum_{i = 1}^N f(x_i) \prod_{j \neq i} \left( \frac{x - x_j}{x_i - x_j} \right) =: q(x). 
  \end{align*}
  The expression above proves that we have $q(x) = \det(xI-A)$ for every $x \neq x_i$. 
  Recalling that $\det(xI-A)$ is a degree $N$ polynomial we can conclude that the relation
  must hold for every $x \in \mathbb{C}$. 
In particular, for every $i = 1, \dots, N$,
we know $q(x_i) = f(x_i)$ and so we can conclude that $q \equiv f$. 
\end{proof}

Since we have that $A$ is a companion matrix for $f$, we can state the following.

\begin{cor} \label{cor:gerschgorin}
  Let $x_1, \dots, x_N\in\CC$ be distinct and let $w_i$ be the values defined above. The union of the balls $B_i = B(x_i, N\cdot\lvert w_i \rvert)$ contains all the roots of $f(x)$. Moreover, every connected component 
  of this union made up of $k$ balls contains exactly $k$ roots. 
\end{cor}

One may wonder which one between Gerschgorin radii and Newton radii is more strict. 
It turns out that if each $x_i$ is a good approximation of a simple root $\zeta_i$ (or more
formally, when $x_i \to \zeta_i$), we have that they are exactly the same and are both 
equal to $N\cdot\lvert x_i - \zeta_i \rvert$ up to the first order expansion. 
This can be verified by the following simple computation under
the assumption that $x_i = \zeta_i + \epsilon_i$ 
with $\epsilon_i$ sufficiently small. 
In the cases of Gerschgorin radii, we have
\[
  N\cdot \left\lvert \frac{f(x_i)}{\prod_{j \neq i} (x_i - x_j)} \right\rvert = 
    N \cdot \lvert x_i - \zeta_i \rvert \cdot \left\lvert \prod_{j \neq i} \frac{x_i - x_j - \epsilon_j}{x_i - x_j} \right\rvert \approx N\cdot\lvert x_i - \zeta_i \rvert 
\]
while, for Newton radii, we obtain
\[
    N\cdot\left\lvert \frac{f(x_i)}{f'(x_i)} \right\rvert = N\cdot\left\lvert
      \frac{1}{ \sum_{j = 1}^n \frac{1}{x_i - x_j - \epsilon_j} } \right\rvert
      \approx N\cdot\left\lvert
      \frac{1}{ \frac{1}{x_i - x_i - \epsilon_i} } \right\rvert = 
      N\cdot\lvert x_i - \zeta_i \rvert. 
\]

This implies that not only they are essentially equivalent, but also that they
are optimal up to a factor of $N$.

\subsection{Detecting Real Roots} 

For a polynomial $f$ with real coefficients, it was observed in 
\cite{2010arXiv1011.1091H} that $\alpha$-theory can be used
to certifiably decide the reality or nonreality of the associated root
given an approximate root.  This local certification test is based on the simple
fact that a root of $f$ is real if and only if there is a real approximate root.

Corollary~\ref{cor:gerschgorin} can be used to develop a global certification test
as developed in the following which uses $\re(z)$ and $\im(z)$ 
to denote the real and imaginary part, respectively, of a number $z\in\CC$.

\begin{thm}  \label{thm:mpsolve_real_classification}
Let $f$ be a monic univariate polynomial of degree $N\geq4$ with 
$N$ distinct roots and $x_1,\dots,x_N\in\CC$.
Suppose that $1\leq i\leq N$ such that for all $j\neq i$,
\begin{enumerate}
\item $|\im(x_i)| < N\cdot\beta(f,x_i)$,
\item $\max\{|\re(x_i) - \re(x_j)|,|\im(x_i)+\im(x_j)|\} > N\cdot(\beta(f,x_i) + \beta(f,x_j))$, and
\item $B(x_i,N\cdot\beta(f,x_i)) \cap \bigcup_{k\neq i} B(x_k,N\cdot\beta(f,x_k)) = \emptyset$.
\end{enumerate}
Then, $f$ has a real root $\zeta_i$ such that $\re(x_i) - N\cdot\beta(f,x_i) \leq \zeta_i \leq \re(x_i) + N\cdot\beta(f,x_i)$.
\end{thm}  
\begin{proof}
To prove by contradiction, suppose that $\zeta_i$ is the unique root of $f$ in the ball 
$B(x_i,N\cdot\beta(f,x_i))$ such that $\im(\zeta_i) \neq 0$.  
Hence, the conjugate of $\zeta_i$, namely $\overline{\zeta_i}$,
is also a root of $f$.  So, there must exists $j$ such that
$\overline{\zeta_i}\in B(x_j,N\cdot\beta(f,x_j))$.  
Observe that Condition $2$ implies that every circle centered in $x_j$, $j \neq i$
cannot intersect the circle $B(\overline{x_i}, N\cdot\beta(f,x_i))$. 
In particular, this implies that 
the complex conjugate of $\zeta_i$ cannot be a root of $f$, 
unless $\zeta_i$ is a real root. 
\end{proof}

In the random setting where, with probability $1$, the univariate polynomial $f$
of degree $N$ will have $N$ distinct roots, we are trying to yield a 
set of $N$ points consisting of approximate roots with distinct associated roots.
When the roots have a large pairwise separation, 
numerical methods such as those implemented in MPSolve yield an 
effective tool for approximating such points.  As the minimum of the 
pairwise distance decreases, one needs to use higher precision in order
to yield approximations to separate the roots and achieve quadratic convergence
of Newton's method.  In the limit when the minimum pairwise distance is zero,
standard Newton's method can only achieve linear convergence.
In this case, for example, MPSolve can compute an approximation correct
to $k$ digits, where the integer $k$ is selected by the user.

\subsection{Certifiable Regions for Comparing the Methods}

As a comparison of the two certification methods, we consider
a ``Wilkinson-type'' polynomial $f(x) = \prod_{j=1}^{10}(x-j)$ 
which has exact integer roots in order to compare numerical results.
Figure~\ref{fig:alphaconv} summarizes results from $\alpha$-theory
and Figure~\ref{fig:convmpsolve} summarizes results from MPSolve.
In both of these plots, the white regions correspond to certifiable regions
where the black points are approximations computed by the
{\tt roots()} command in {\sc Matlab} and MPSolve, respectively.
When using MPSolve, the roots are inside the certifiable regions
while the {\tt roots()} command, probably due to the limitation of using 
double precision computations, obtained roots outside of some $\alpha$-theoretic
certifiable quadratic convergence basins.

\begin{figure}[ht]
  \begin{center}
    \includegraphics[width=\linewidth]{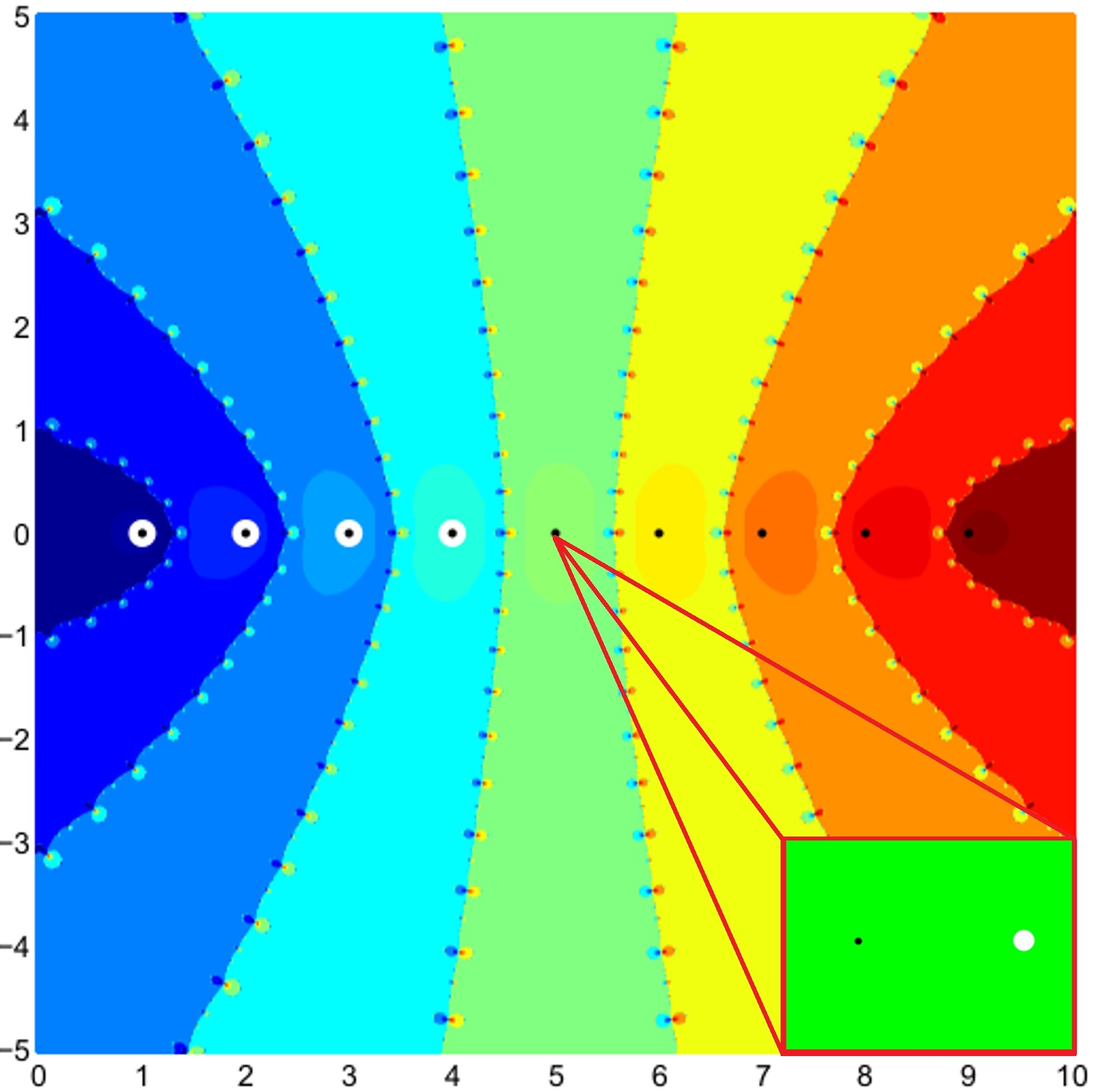}
   \end{center}
   \caption{Certifiable regions obtained using {\tt alphaCertified} with the certified
   quadratic convergence zone in white and the approximations from the {\tt roots()} command in {\sc Matlab}
   in black.}
   \label{fig:alphaconv}
\end{figure}

\begin{figure}[ht]
  \begin{center}
    \includegraphics[width=\linewidth]{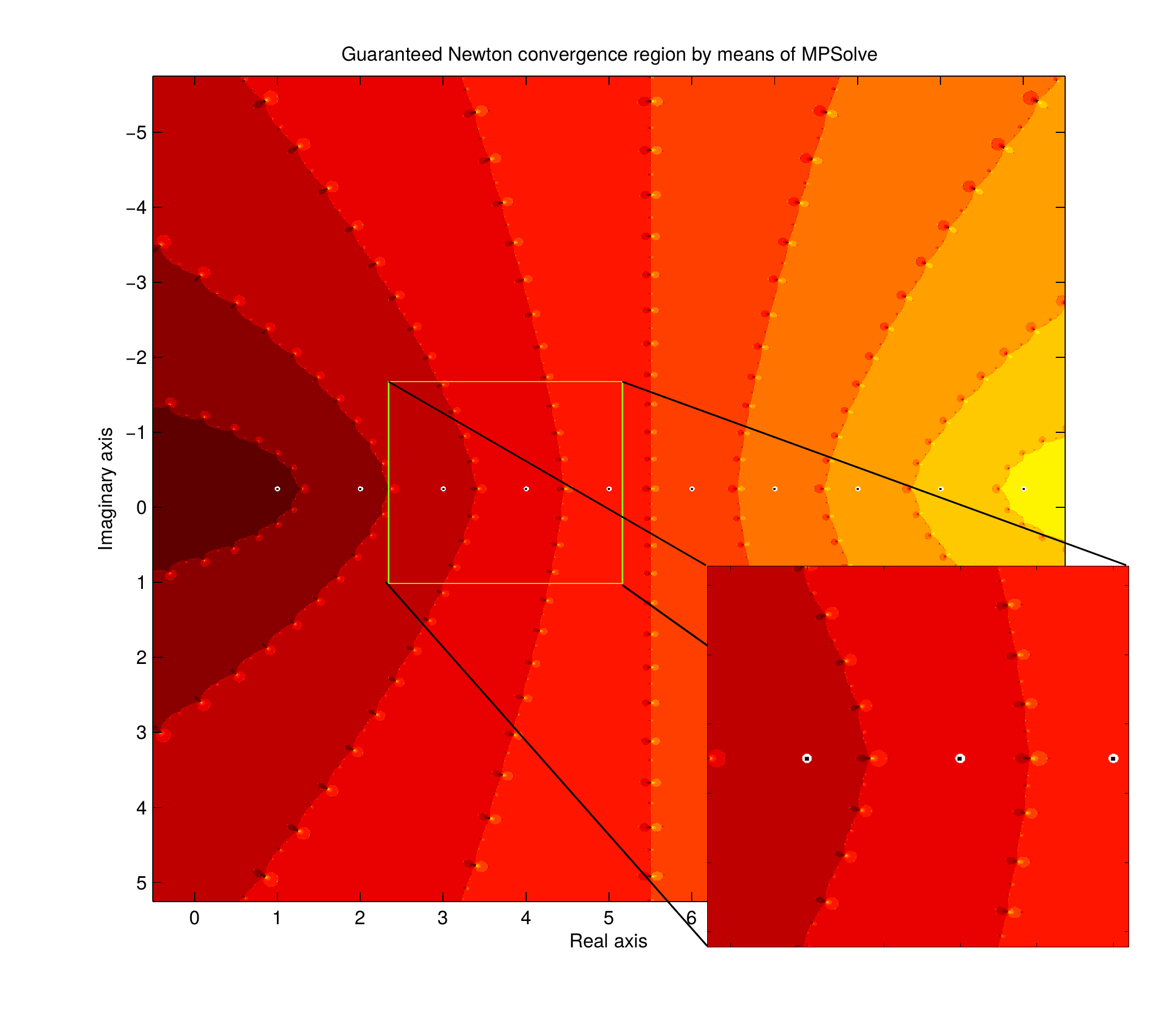}
  \end{center}
  \caption{Certifiable regions obtained by means of Tilli's Theorem. The
   black dots are the approximations computed using MPSolve.}
  \label{fig:convmpsolve}
\end{figure}

In the rest of the figure the areas were filled with different colors based on the
root to which the Newton method converges if using that point as the starting~one. 

We can note how the regions obtained through MPSolve's certification procedure have
a similar radius, while this is not true for {\tt alphaCertified}. This is a direct
consequence of the uniformity of the constants involved in the first strategy. 

\section{Numerical Setup}\label{sec:numerical_set_up}
The following describes our numerical setup for generating random
polynomials.  In particular, we generate a random polynomial of 
degree $N$ by drawing the coefficients from a chosen random distribution such as the uniform distribution or the Gaussian distribution with specified mean and variance using {\sc Matlab} which 
generates random numbers with 32-bit precision. 
Then, we use two numerical root-finding methods and two certification methods.

First, we approximate all of the roots using either the 
companion matrix method as implemented in 
{\sc Matlab}’s {\tt roots()} command\footnote{We have tried other root-finding methods such as Lindsey-Fox algorithm, as well 
the built-in root-finding methods in Maple and Mathematica. 
However, we find the same conclusions as with {\sc Matlab}'s {\tt roots()} command.} 
or using MPSolve. 

To use $\alpha$-theory to certify the roots, we compute 
a set of $N+1$ rational approximations $\hat{a_i}$ of the coefficients
of the polynomial $f(x) = \sum_{i = 0}^N a_i x^i$
such that, for each $i = 0, \dots, N$ we have $\lvert a_i - \hat{a_i} \rvert \leq \lvert a_i \rvert \cdot 10^{-16}$. 
With numerically estimated roots and rational coefficients in hand, we use
{\tt alphaCertified} using $256$-bit floating point arithmetic
to certify the approximated roots.  
If any of the $N$ roots is not certified, then we 
run several Newton iterations on the uncertified root until 
it finally reaches the quadratic convergence zone of some exact root.
In our experiments, we obtained~$N$ certified distinct roots for 
the random sample and then counted the real roots.

As an alternative approach, we use the approach presented in 
Section~\ref{subsec:mpsolve} to obtain approximate solutions for 
the roots of a given polynomial. 
MPSolve provides a {\sc Matlab} interface 
which allows the user to request an arbitrary number of digits
and certify the result.  This approach is valuable 
because it alleviates the numerical ill-conditioning issues that 
may be encountered when solving high degree polynomials with (possibly big) 
integer coefficients.
For these cases, the {\sc Matlab}'s {\tt roots()} command, 
even if using a backward stable algorithm, 
can give very poor approximations of the roots. 
As a consequence, it is possible that {\tt alphaCertified} 
may not be able to certify them due to large errors.
This issue is solved by MPSolve 
using a global approach that is guaranteed to produce 
$16$ correct digits on the output (and certify they are approximate solutions). 
A suitable level of floating point multiprecision will be used by MPSolve internally
to achieve this, but it may be not necessary to expose this to {\sc Matlab} (unless more than $16$ digits are required). 

This method is effective in our experiments. 
As an example, we solved $2000$ random polynomials of degrees 
between $100$ and $1000$ using MPSolve, which was
able to certify all the roots and determine their reality.
While in principle it may happen that a root cannot be classified via
Theorem~\ref{thm:mpsolve_real_classification}, this did 
not happen in our experiments. 

\section{Results}\label{sec:results}

The following presents our results for the roots of 
random polynomials with coefficients taking 
indenpendent and identically distributed (i.i.d.) 
values from two distributions:
the Gaussian distribution with mean $0$ and variance $1$ which is one of the most extensively studied cases and can serve as our benchmark as many analytical results are available here, and the Cauchy distribution for which only a handful of analytical results are available and hence
our numerical results can provide valuable input to the theory.
 
\subsection{Quantifying Certification Procedures}

While trying to approximate the roots, we
wished to ensure that our numerical approximations
were in the quadratic convergence zone of Newton's method.
In Section~\ref{sec:Numerical_Methods}, two different certification methods are discussed.
The first, based on $\alpha$-theory, 
which requires that $\alpha(f,x) < \left(13-3\sqrt{17}\right)/4$,
is a local certification approach in that it requires only
data from one point.  
The second, based on Gerschgorin's theorem, 
which requires that the Newton correction is smaller than the 
relative distance between the roots divided by approximately $6N$,
is a global certification approach in that it requires information
about all roots simultaneously.

The second approach can be reformulated in the same framework
as $\alpha$-theory if we replace $\gamma(f,x)$ with 
$\gamma_M(f) = 6N \cdot \max_{i \neq j} (\zeta_i - \zeta_j)^{-1}$.
One can now compare the mean value of $\gamma$ at the roots of $f$ 
with $\gamma_M$ with the smaller value leading to a 
larger~certification~region.

If we assume that the distance of the roots is proportional to $N^{-1}$,
which is a realistic expectation for random polynomials, 
then $\gamma_M \approx O(N^2)$. 
Experimentally, we have verified that $\gamma(f,x)$ in 
the $\alpha$-theory has an exponential growth, as can be
clearly seen in Figure~\ref{fig:meangamma}. 
This causes certification to be computationally prohibitive
as $N$ increases.  For example, using only double precision, 
Figure~\ref{fig:meangamma} suggests that degrees less than $100$
are feasible using double precision, but may be insufficient
for degrees larger than $100$ due to large values of $\gamma$.

\begin{figure}[ht] 
  \begin{center}
    \includegraphics[width=.9\linewidth]{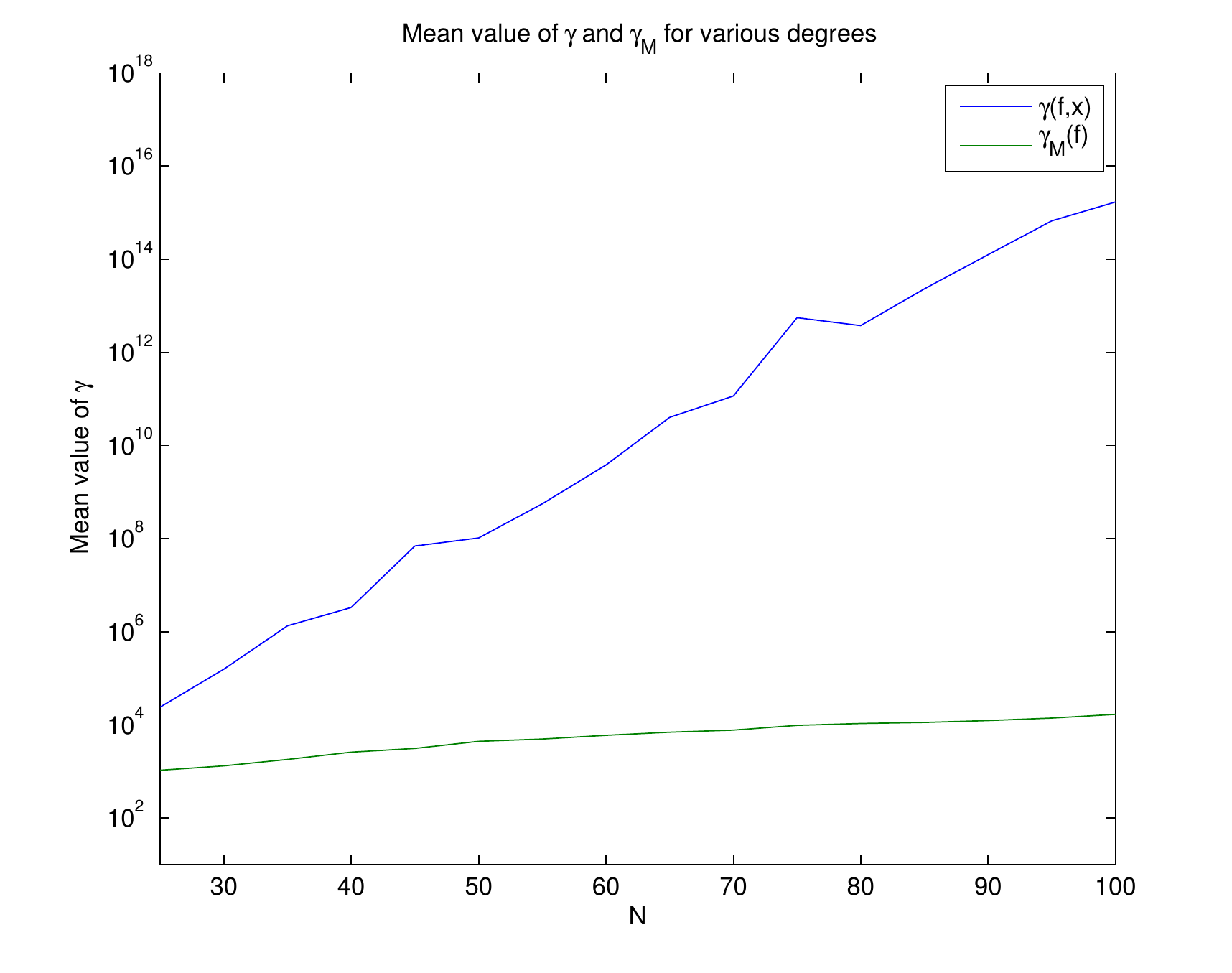}
  \end{center}
  \caption{The mean value of the $\gamma(f,x)$ and $\gamma_M(f)$ constants
   for random Cauchy polynomials with degree between $25$ and $100$.}
  \label{fig:meangamma}
\end{figure}

It is important to note that the comparison between these two strategies is not fair since $\gamma(f,x)$ is directly computable 
given a single approximation while $\gamma_M(f)$ needs to be 
bounded using some kind of inclusion theorem since it 
relies on the (unknown) distance between roots of $f$. 
However, our numerical results should provide valuable feedback 
when comparing the two certification procedures.
In particular, our experiments showed that the usage of
the second strategy, using a global approach, is much more 
effective for the cases under consideration. 

In Section~\ref{subsec:mpsolve}, we showed that the 
Gerschgorin radii are suboptimal only up to a factor 
of $N$.  This additional factor of $N$ leads to a 
moderate overestimate of~$\gamma_M$, namely $\gamma_M \approx O(N^3)$. 

\subsection{Real Roots}

For a first experiment, we analyzed the 
number of real roots of random polynomials. 
Many theoretical results are known, 
but further analysis could be carried out to obtain
more information about the real roots. 
In particular, the following summarizes some
experimental results that we have obtained
by solving $2000$ polynomials of 
degrees $100$, $200$, \dots, $1000$. 
We measured the mean and variance for our sample
and checked our results against the theoretical predictions, when available.

\subsubsection{Average Number of Real Roots}

The average number of real roots is one of the most well-studied quantities in the
context of random polynomials. We have estimates both for the Gaussian case and for
the Cauchy case. We only summarize the main results here and leave
more in-depth information to \cite{Farahmand:98}. 

\begin{thm} Let $f(x) = \sum_{i = 0}^N a_i x^i$ be 
a random polynomial of degree $N$ where $a_i$ are i.i.d. 
selected from a standard normal distribution.
If $R(a_0, \dots, a_n)$ is the number of real roots of $f$,
there exists a constant $C_0$ such that
  \[
    \mathbb{E}(R(a_0, \dots, a_n)) = \frac{2}{\pi} \log N + C_0 + \frac{2}{N \pi} + 
      O\left( \frac{1}{N^2} \right).
  \]
The constant $C_0$ is approximately equal to $0.6257358072$. 
\end{thm}

\begin{thm} 
  Let $f(x) = \sum_{i = 0}^N a_i x^i$ be a random polynomial of degree $N$ where   $a_i$ are i.i.d. with a standard Cauchy distribution. 
If $R(a_0, a_1, \dots, a_N)$ is the number of real roots of $f$,
then there exists constants $C$, $A_0$, and $A_2$ such that
  \[
    \mathbb{E}(R(a_1, \dots, a_N)) = C\log (N + 1) + A_0 + \frac{A_2}{(N+1)^2} + 
      O\left( \frac{1}{N^3} \right). 
  \]
The constants $C$, $A_0$, and $A_2$ are approximately equal to: 
\[
C = 0.7413, \qquad
A_0 = 0.559132, \qquad
A_2 = 0.230596.
\]
\end{thm}

\begin{figure}[ht]
  \includegraphics[width=.8\linewidth]{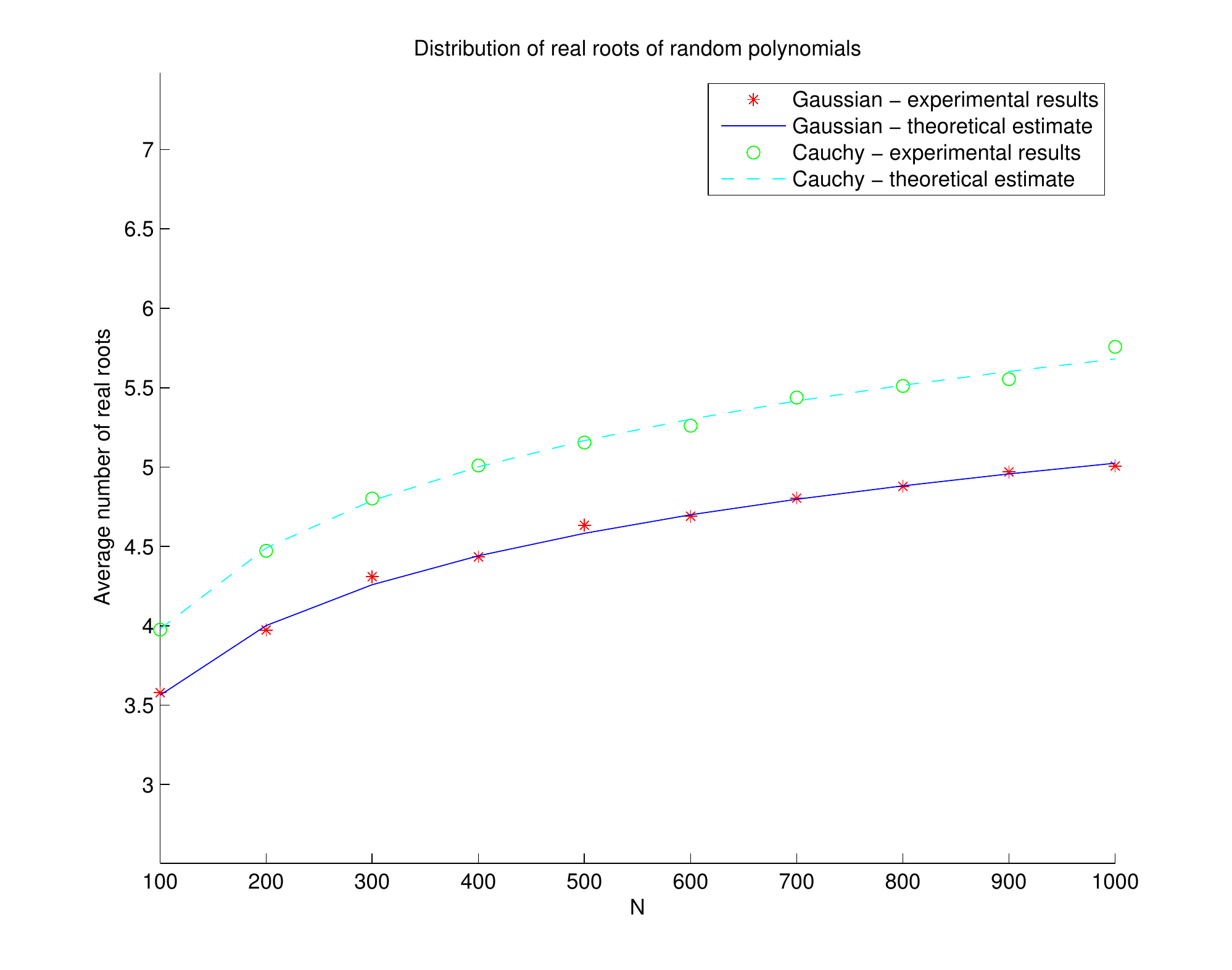}
  \caption{The mean number of real roots for Cauchy and Gaussian random polynomials.}\label{fig:cauchy_gauss_mean}
\end{figure}

The experimental results are summarized in 
Figure~\ref{fig:cauchy_gauss_mean}. 
This figure shows that the theoretical results agree with the 
experimental results, but the convergence can be quite slow. 
That is, even with $2000$ sample polynomials, 
there are still some oscillations 
around the theoretical value that are visible from the plot. 



\subsubsection{Variance of the Number of Real Roots}

From the experimental data, we can estimate the variance
on the number of real roots.
For the Gaussian distribution, we compare this with the 
following theoretical result obtained in \cite{maslova}.

\begin{thm} \label{thm:maslova}
 Let $f(x) = \sum_{i = 0}^N a_i x^i$ be a random polynomial such that 
  the random variables $a_i$ are i.i.d. and satisfy the following: 
  \begin{itemize}
    \item $\mathbb{E}(a_i) = 0$, and
    \item the random variables $a_i$ do have a moment of order $2+s$ for
      some $s > 0$. 
  \end{itemize}
  If $M(N) := 4 \log(N)\left(\frac 1 \pi - \frac{2}{\pi^2} \right)$,
  then
  \[
    \lim_{N \to \infty} \frac{\mathrm{Var}(\#\{ \text{real roots of } f(x) = 0\})}{M(N)}  = 1.
  \]
\end{thm}

\begin{figure}[ht]
  \begin{center}
    \includegraphics[width=\linewidth]{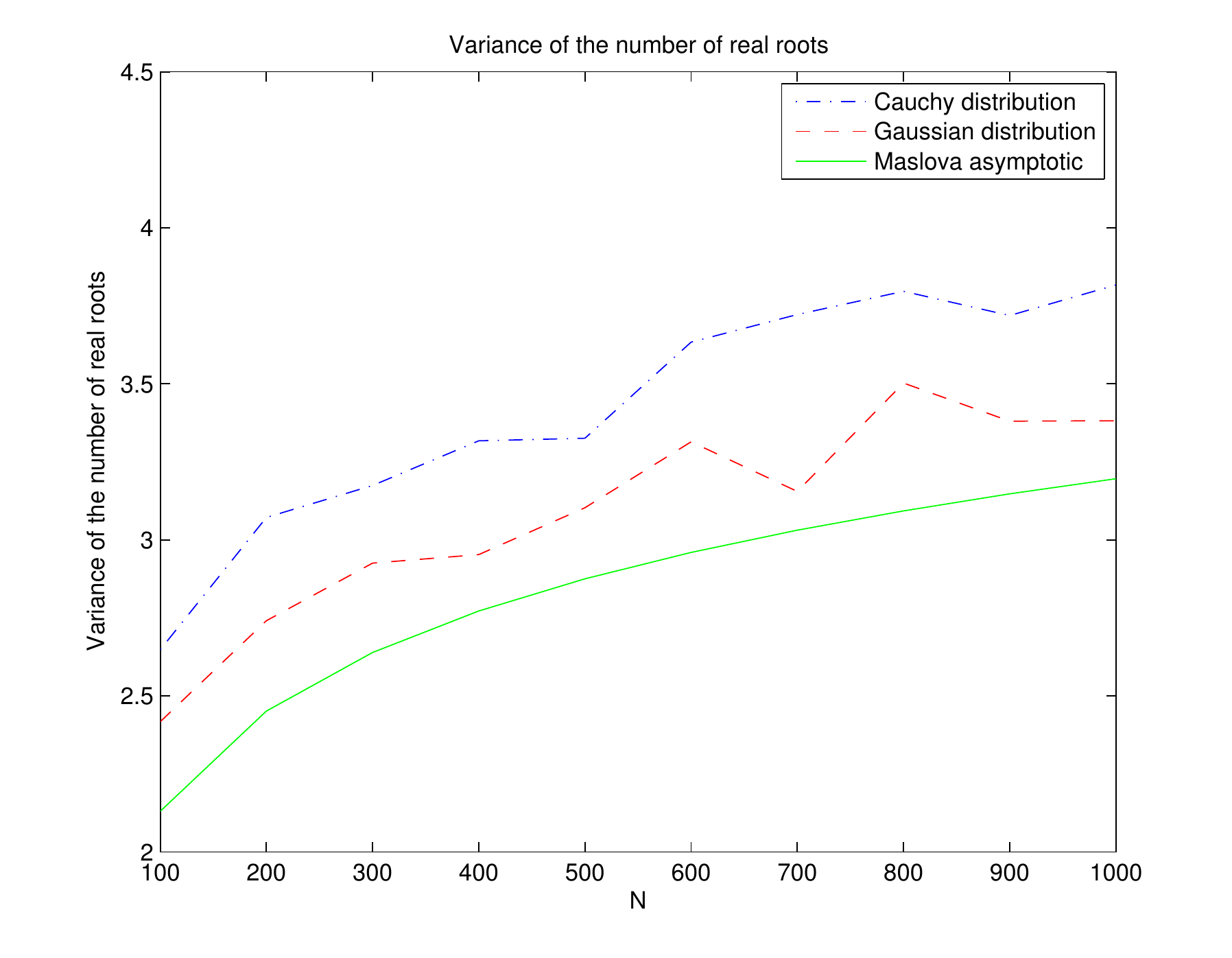}
  \end{center}
  \caption{The variance of the number of real roots for Gaussian random
  polynomials and Cauchy random polynomials.}
  \label{fig:variance}
\end{figure}

We can see from Figure~\ref{fig:variance} that the 
experimental estimate and theoretical value seems to be off by a constant. This is due to the fact that, unlike in the case of the expected value, we
do not know all the terms of the expansion so we can only state that the limit of 
$\mathrm{Var}(\# \{ \text{real roots of } f(x) \} ) / M(N)\rightarrow1$
as $N\rightarrow\infty$.

Moreover, the plot seems to confirm that both the Gaussian and the Cauchy case
have a logarithmic behavior as $N$ increases. Note that the hypotheses of 
Theorem~\ref{thm:maslova} are not satisfied in case of the Cauchy random polynomial
since the random variables do not have moments of any order. 




\subsubsection{Histograms of the Number of Real Roots}

In Figure~\ref{fig:gauss_no_real_roots_hist} and \ref{fig:cauchy_no_real_roots_hist}
we report the number of real roots for random polynomials of degree
$1000$ in form of an histogram. We analyze both the Gaussian and the Cauchy
case. The plots were generated starting from a sample of $1200$ random
polynomials with the different distributions for the coefficients. 

\begin{figure}[ht] \centering
  \begin{subfigure}[t]{.45\textwidth} \centering
    \includegraphics[width=\linewidth]{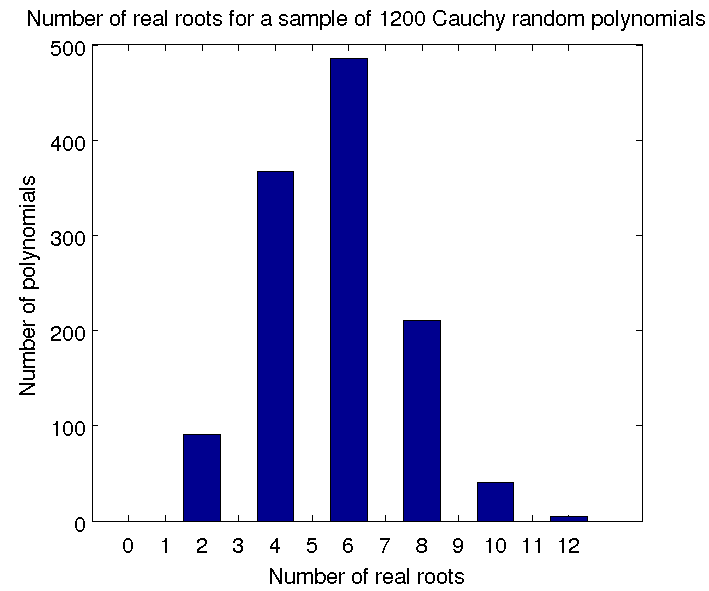}
    \caption{Cauchy polynomials.}
    \label{fig:cauchy_no_real_roots_hist}
  \end{subfigure}
  \begin{subfigure}[t]{.45\textwidth} \centering
    \includegraphics[width=\linewidth]{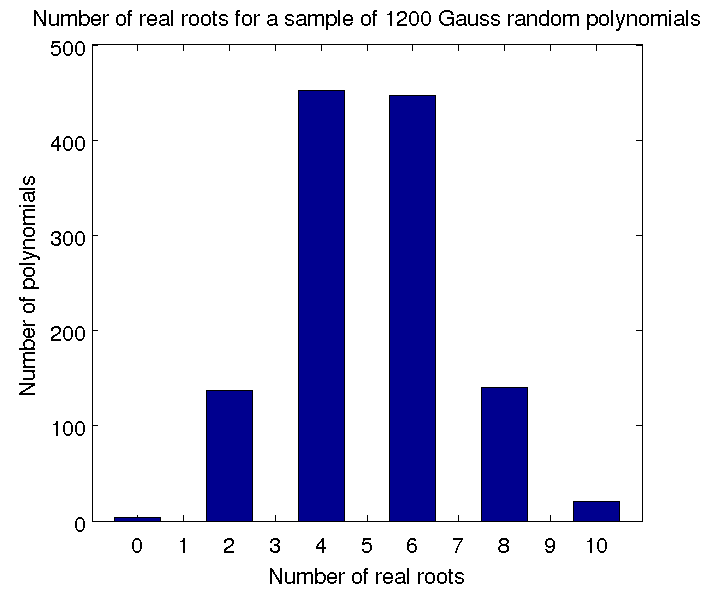}
    \caption{Gauss polynomials.}
    \label{fig:gauss_no_real_roots_hist}
  \end{subfigure}  
  \caption{Distribution of the number of real roots for a sample $1200$ random
    polynomials obtained using Cauchy and Gauss distributions.}
\end{figure}

\subsubsection{Average Number of Positive/Negative Real Roots}
Motivated by the problem of finding the average number of positive and negative real eigenvalues of real matrices, which amounts to finding
the average number of positive and negative real roots of the corresponding characteristic polynomials, we ask what is the average number of 
positive and negative real roots for our random polynomials.
We expect that the mean of the number of positive roots
is equal to the mean number of negative roots
since both of our random distributions
are symmetric with respect to the origin.
That is, it is easy to check that the polynomial
$f(-x) = \sum_{i = 0}^N (-1)^i a_i x^i$ has 
roots the opposites of the ones of $f(x)$ 
but the distribution of the coefficients is the same (since both in the Gaussian and in the Cauchy case multiplying by $-1$ does not change 
the random variable). 

Even though this is clear from an asymptotic point of view, it may be interesting to further examine the topic by looking at the variance of the number of positive and negative roots. 

\begin{figure}[ht]
 \centering
 \includegraphics[width=.8\linewidth]{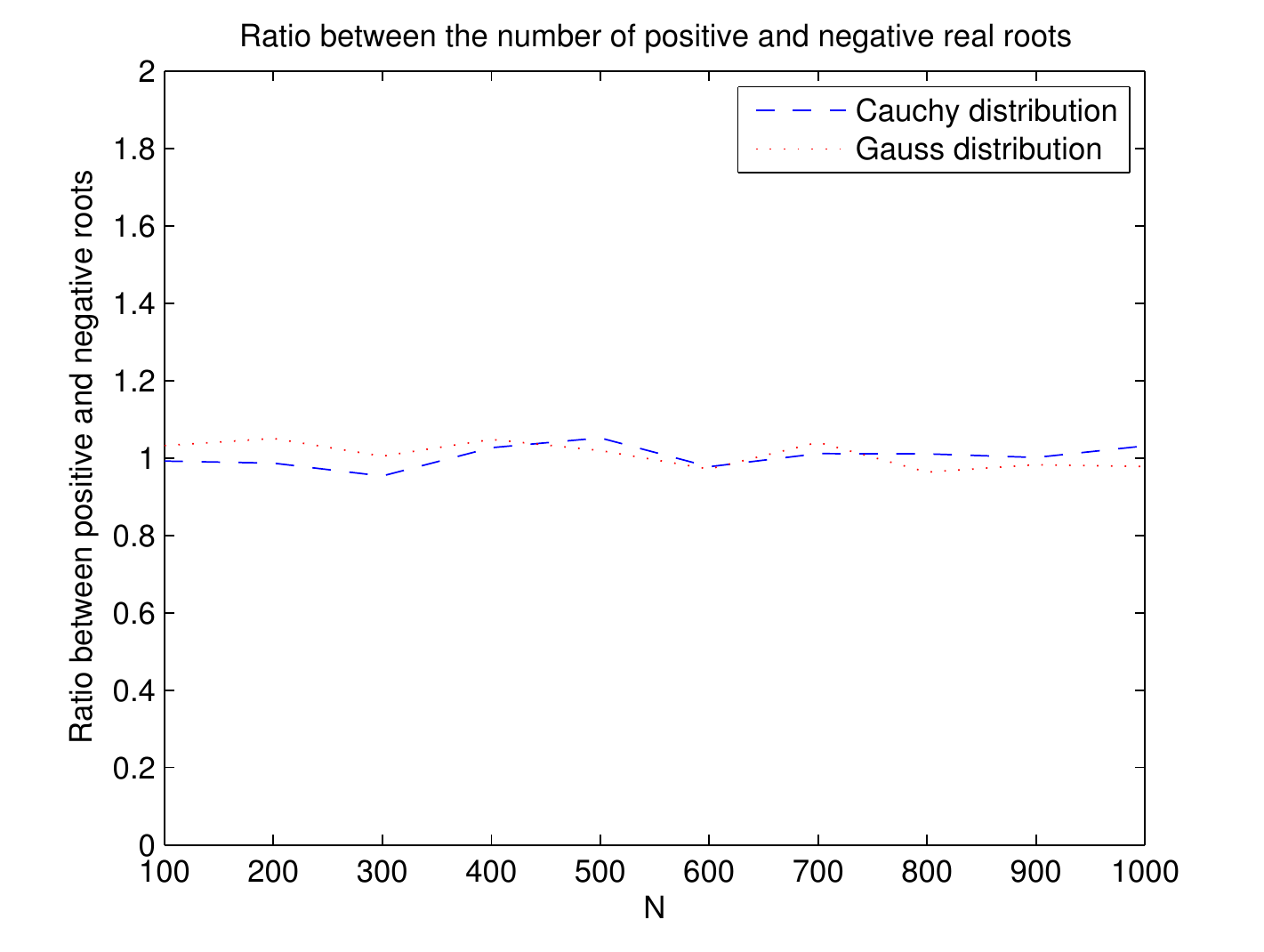}
 \caption{Ratio between the number of positive and negative real roots
 of a small sample of random polynomials.}
 \label{fig:ratio}
\end{figure}

As a further experimental proof of the above statement, one can look at Figure~\ref{fig:ratio}
that shows the ratio between the number of positive and negative real roots for a sample
of random polynomials. We have kept the number of samples low to show that there is 
still some floating around the value $1$, but it is quite clear that the 
problem of counting the roots is symmetric in the change of sign. 

\subsubsection{Variance on the Number of Positive Roots}

\begin{figure}[ht]
  \begin{center}
    \includegraphics[width=.9\linewidth]{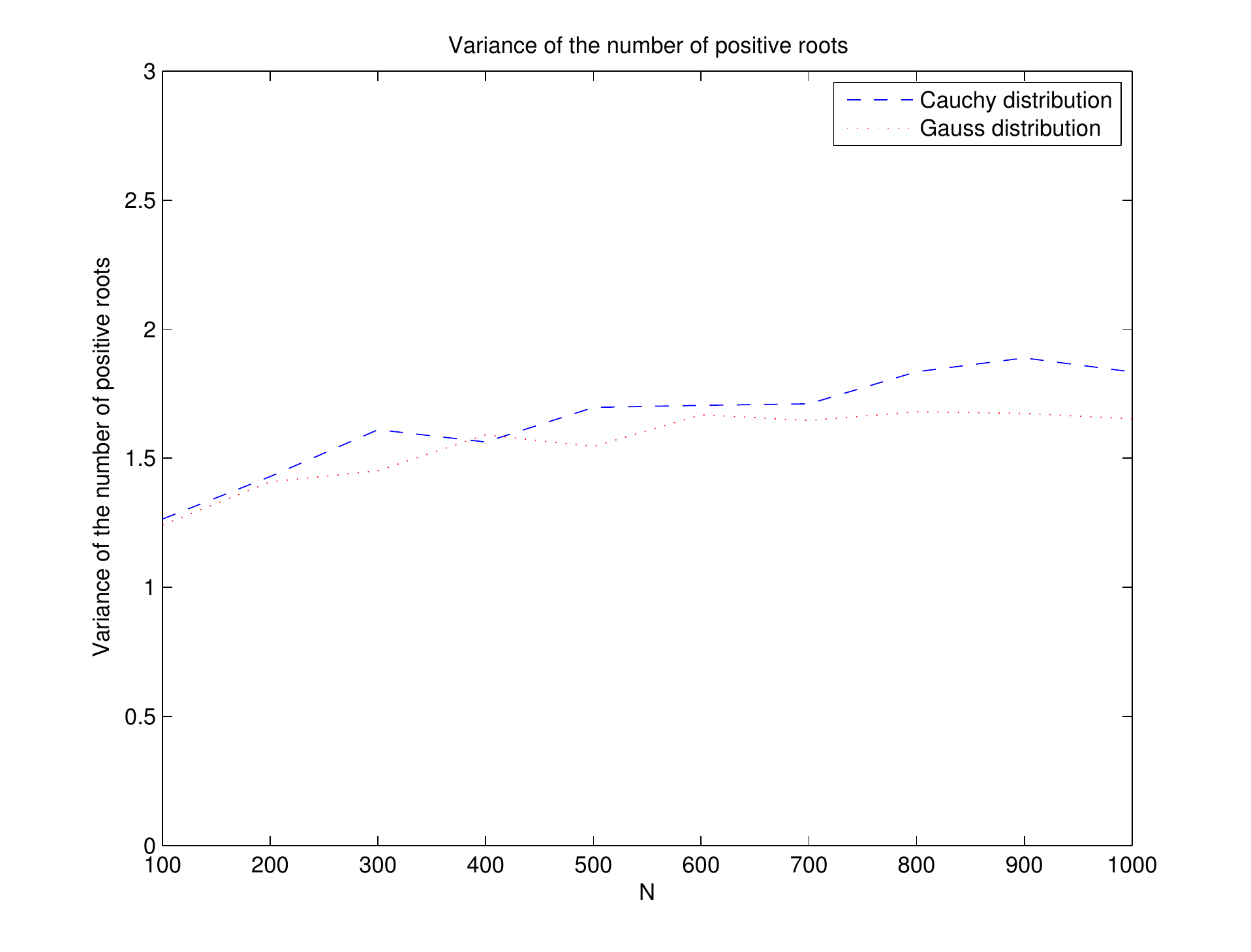}
  \end{center}
  \caption{Variance of the number of positive roots of a random polynomial
   with different random distribution of the coefficients.}
  \label{fig:variancepositive}
\end{figure}

Figure~\ref{fig:variancepositive} plots our experimentally
computed variance on the number of positive real roots
using $2000$ random polynomials for each sample point.
As far as we know, there are no theoretical estimates 
for the variance of the number of positive roots. 
Unfortunately, since the variance is increasing very slowly,
is not easy to guess its~asymptotic~behavior.


\subsection{Complex Roots}

A natural extension of the questions that we have tried to address in the previous section
is the analysis of the distribution of complex roots. We are considering real
random distributions so we expect to have a certain number of real roots and
some complex conjugate ones. 
The following present evidence about the 
distribution of the real and imaginary parts and on their modulus.

\subsubsection{Distribution of the Complex Roots}

Figures~\ref{fig:realdistcauchy}--\ref{fig:imagdistgauss} 
plot the approximate density of the distribution of the 
real and imaginary parts of the roots of Cauchy and Gauss polynomials, respectively. 

\begin{figure}[ht]
  \begin{center}
    \includegraphics[width=\linewidth]{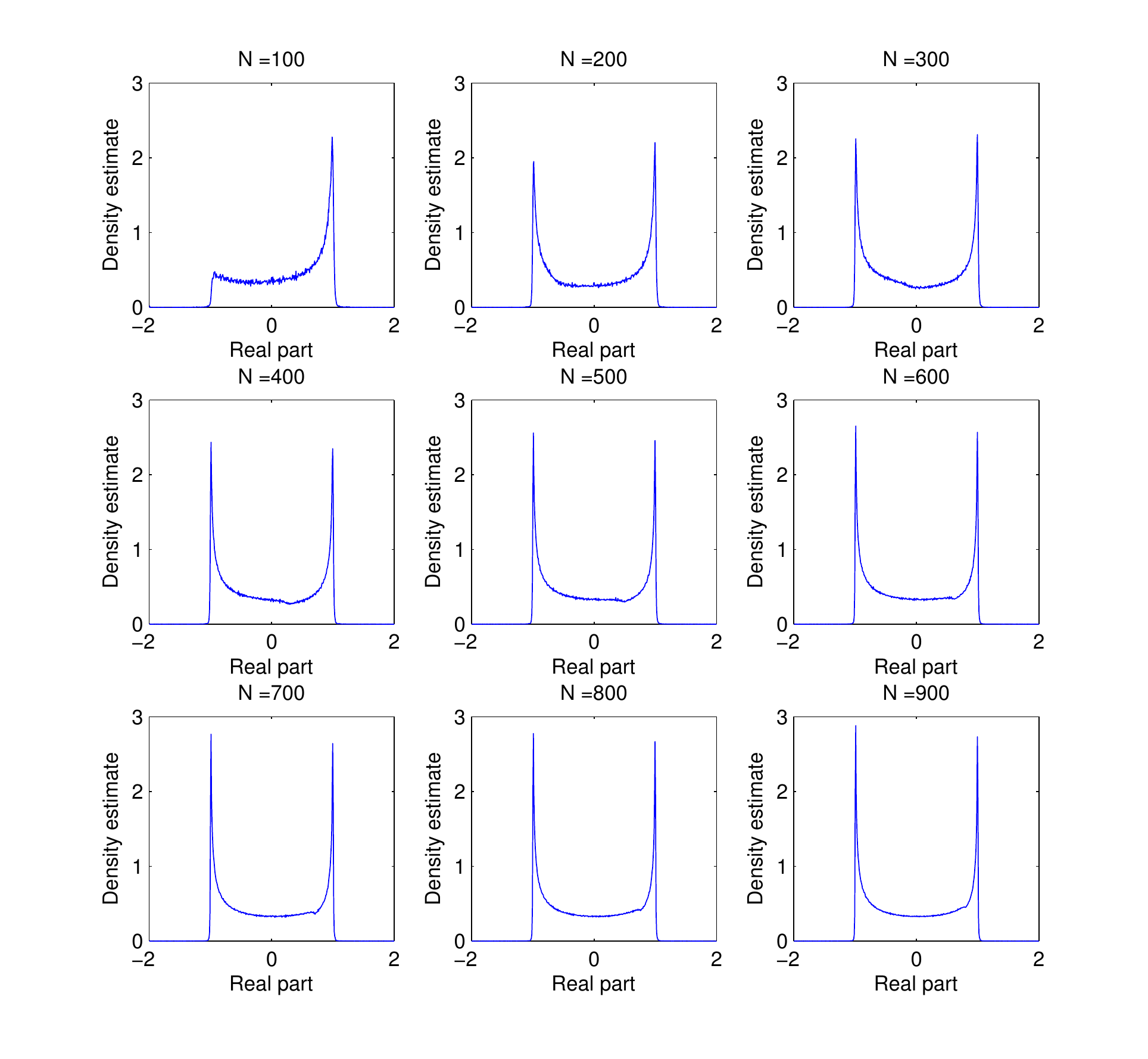}
  \end{center}
  \caption{Density of the distribution of the real part of the roots of random
    Cauchy polynomials for various degrees.}
  \label{fig:realdistcauchy}
\end{figure}

\begin{figure}[ht]
  \begin{center}
    \includegraphics[width=\linewidth]{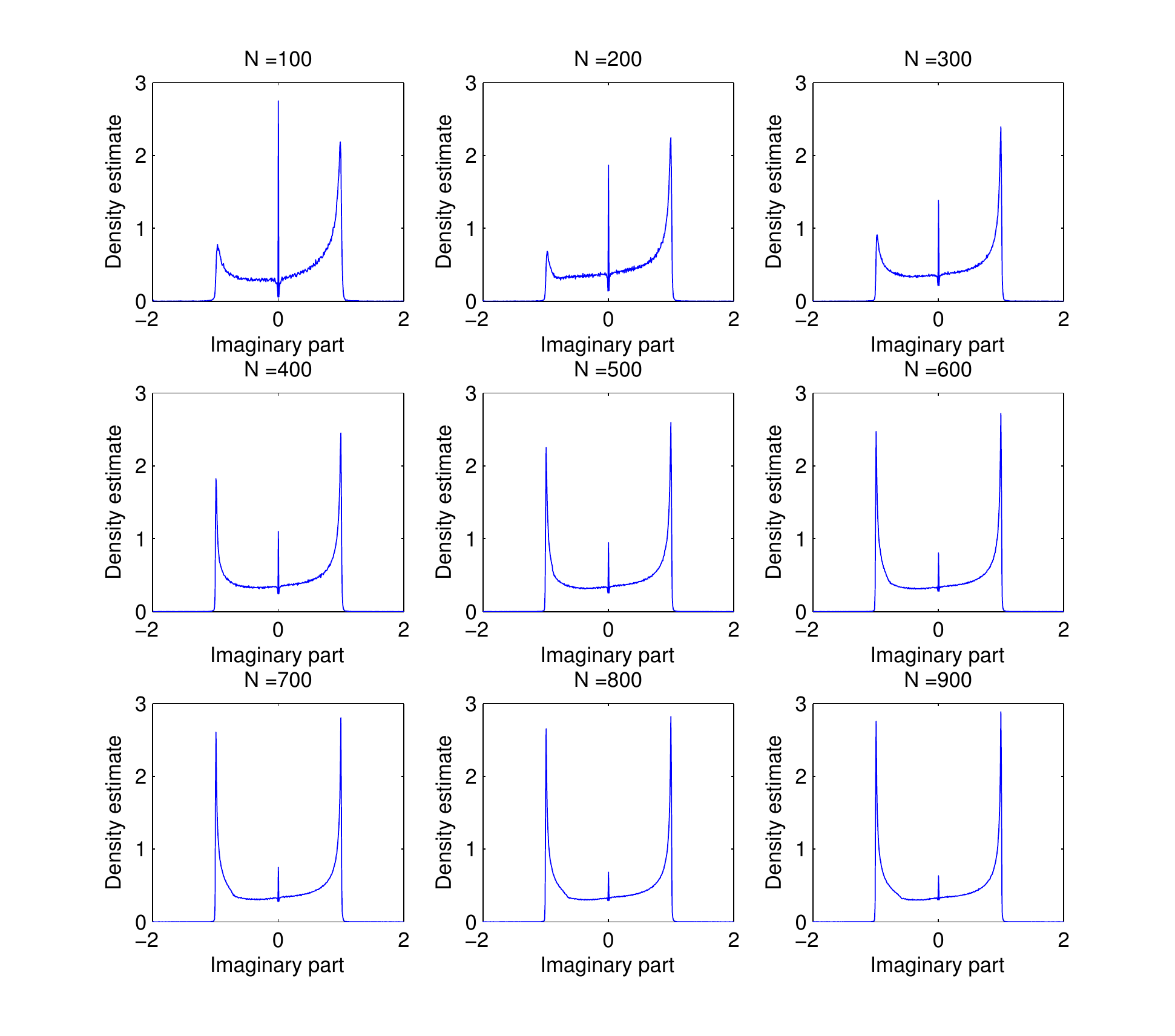}
  \end{center}
  \caption{Density of the distribution of the imaginary part of the roots of random
    Cauchy polynomials for various degrees.}
  \label{fig:imagdistcauchy}
\end{figure}

\begin{figure}[ht]
  \begin{center}
    \includegraphics[width=\linewidth]{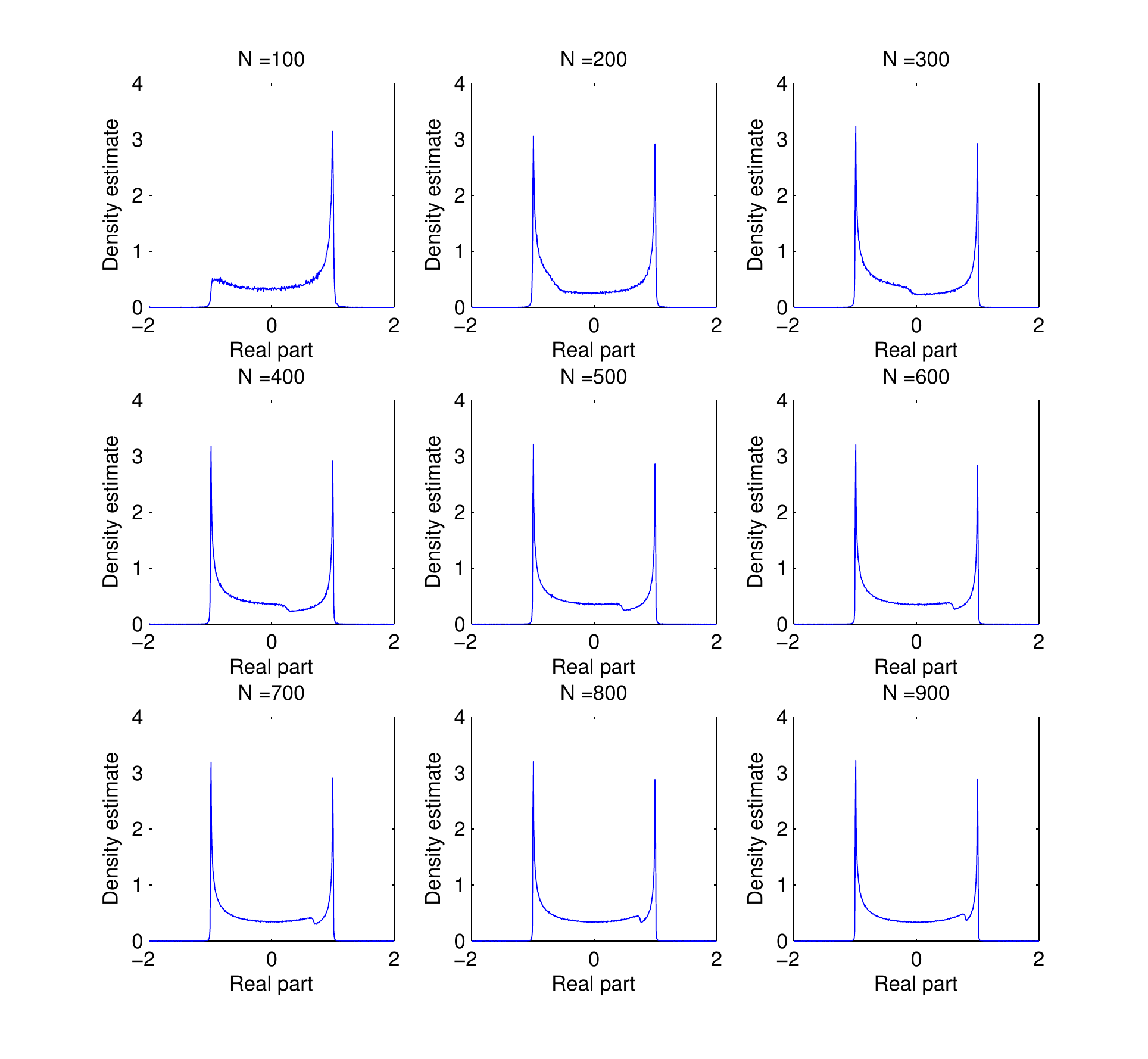}
  \end{center}
  \caption{Density of the distribution of the real part of the roots of random
    Gauss polynomials for various degrees.}
  \label{fig:realdistgauss}
\end{figure}

\begin{figure}[ht]
  \begin{center}
    \includegraphics[width=\linewidth]{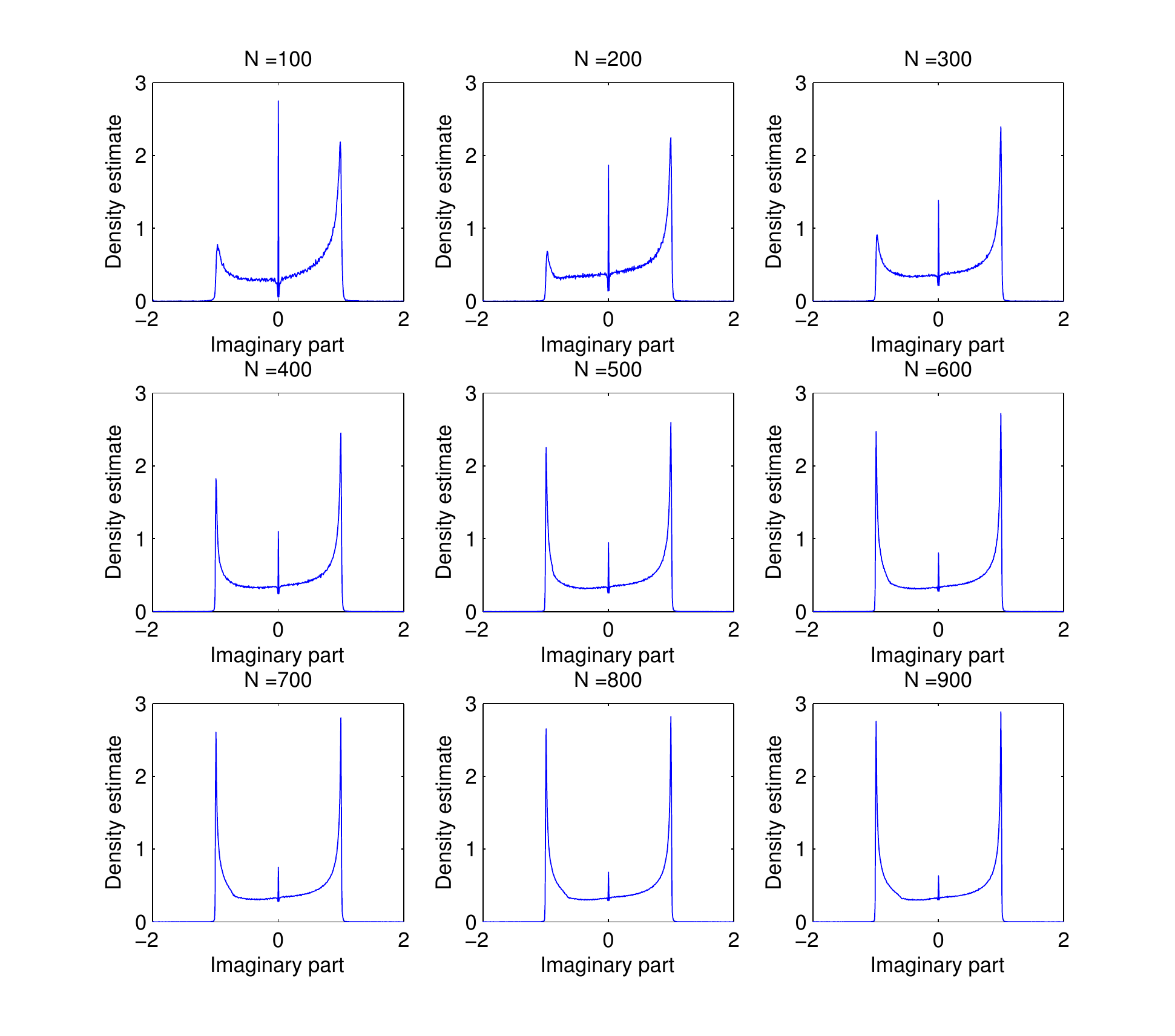}
  \end{center}
  \caption{Density of the distribution of the imaginary part of the roots of random
    Gauss polynomials for various degrees.}
  \label{fig:imagdistgauss}
\end{figure}

\subsubsection{Magnitude of the Roots}

\begin{figure}[ht]
  \centering
  \begin{subfigure}[b]{.45\textwidth} \centering
    \includegraphics[width=\linewidth]{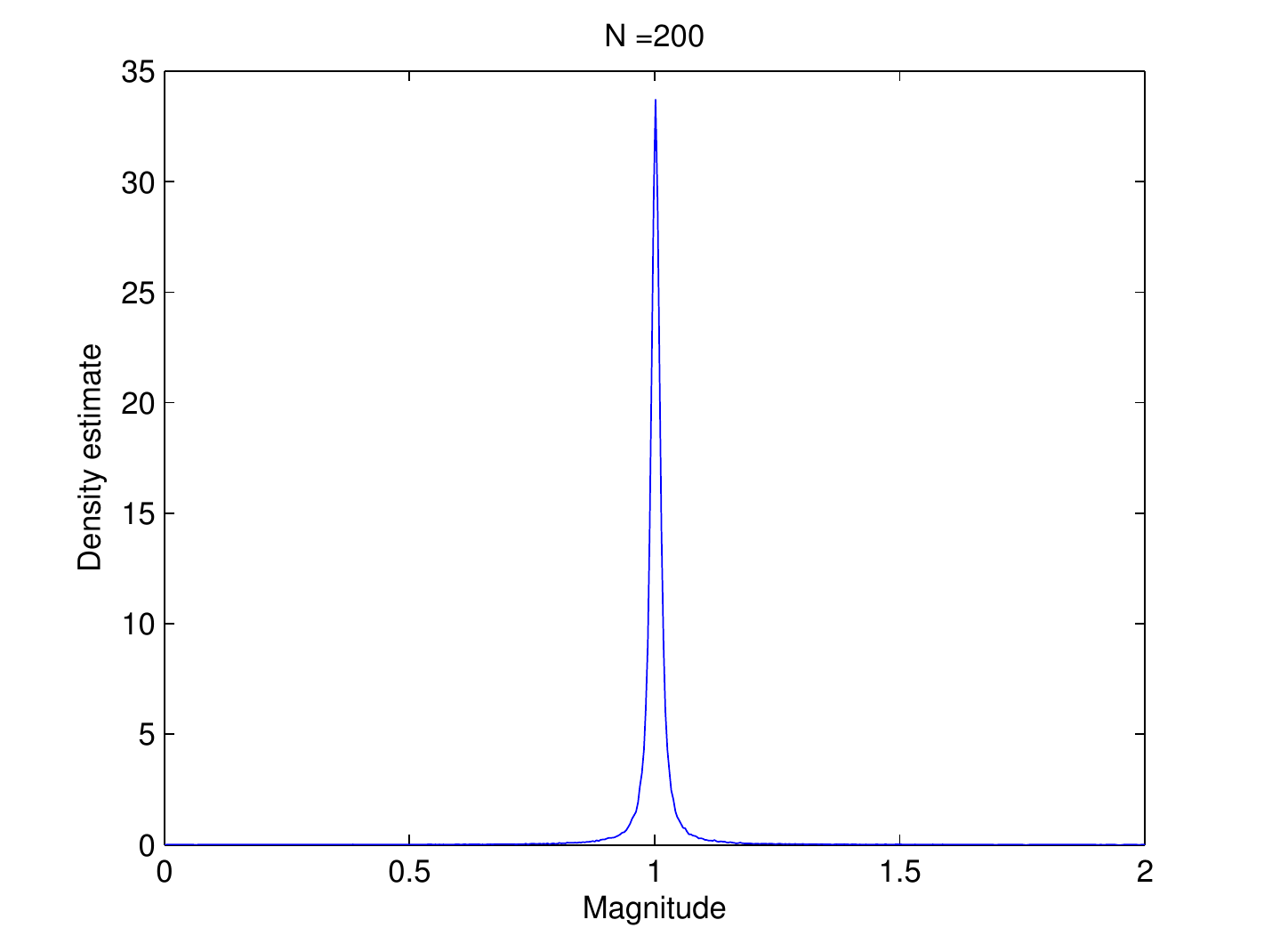}
    \caption{Experimental distribution of the magnitude of the roots of a 
    random Gauss polynomial.}
    \label{fig:gaussdistabs}
  \end{subfigure}
  \hspace{0.1in}
  \begin{subfigure}[b]{.45\textwidth} \centering
    \includegraphics[width=\linewidth]{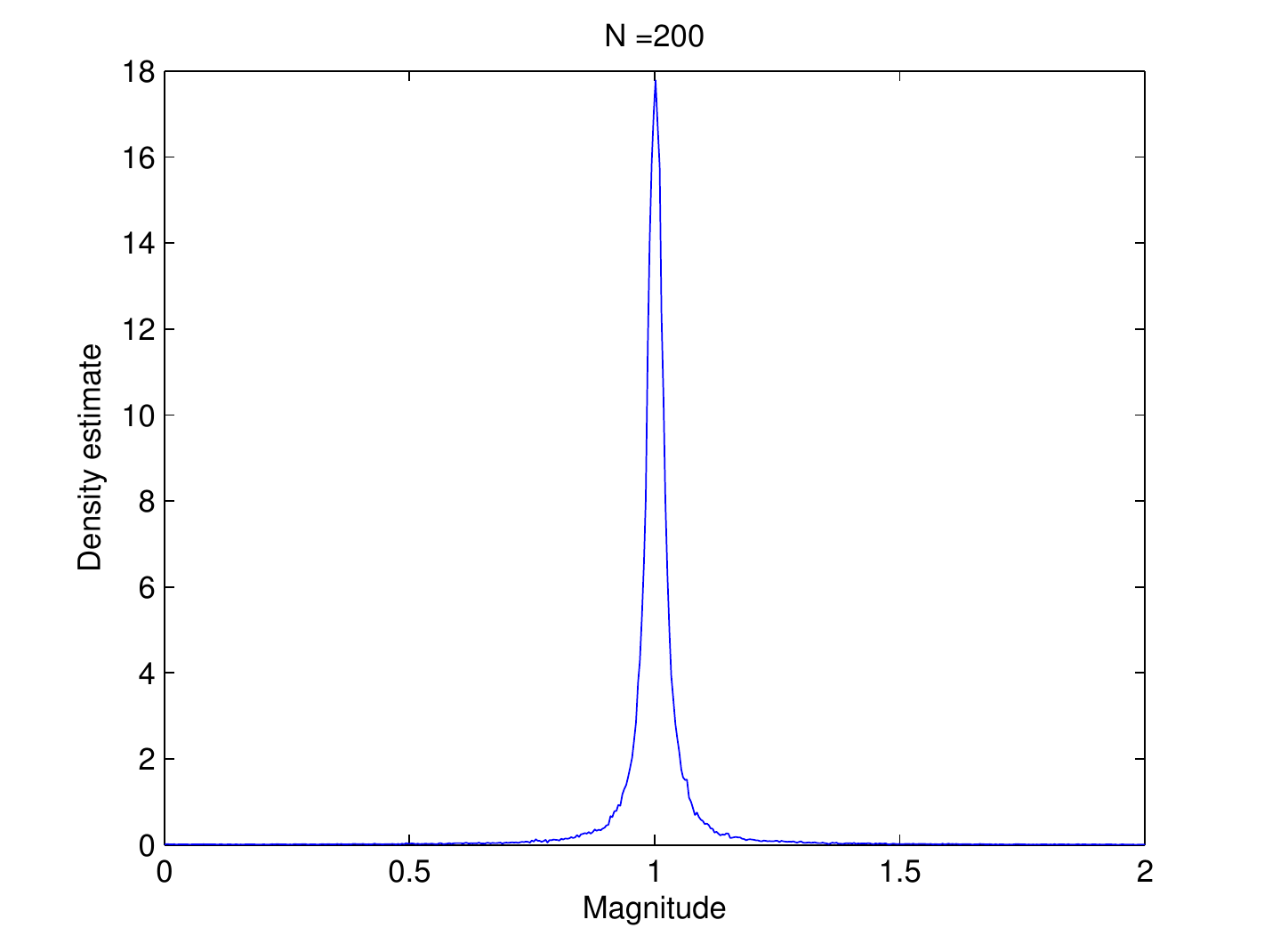}
    \caption{Experimental distribution of the magnitude of the roots of a 
      random Cauchy polynomial.}
    \label{fig:cauchydistabs}
  \end{subfigure} \hfill  
  
  \caption{Magnitude of the roots with different distributions.}
\end{figure}

We also analyze the distribution of the modulus 
of the complex roots. The results of our experiments are 
reported in Figure~\ref{fig:gaussdistabs} and in Figure~\ref{fig:cauchydistabs}. 
Only the case with $N=200$ is reported since
all our plots look very similar and the 
modulus of almost all the roots have modulus near $1$. 

\subsection{Stationary Points of Univariate Random Potential Energy Landscape}

\begin{figure}[ht] \centering
  \begin{subfigure}[t]{.45\textwidth} \centering
    \includegraphics[width=\linewidth]{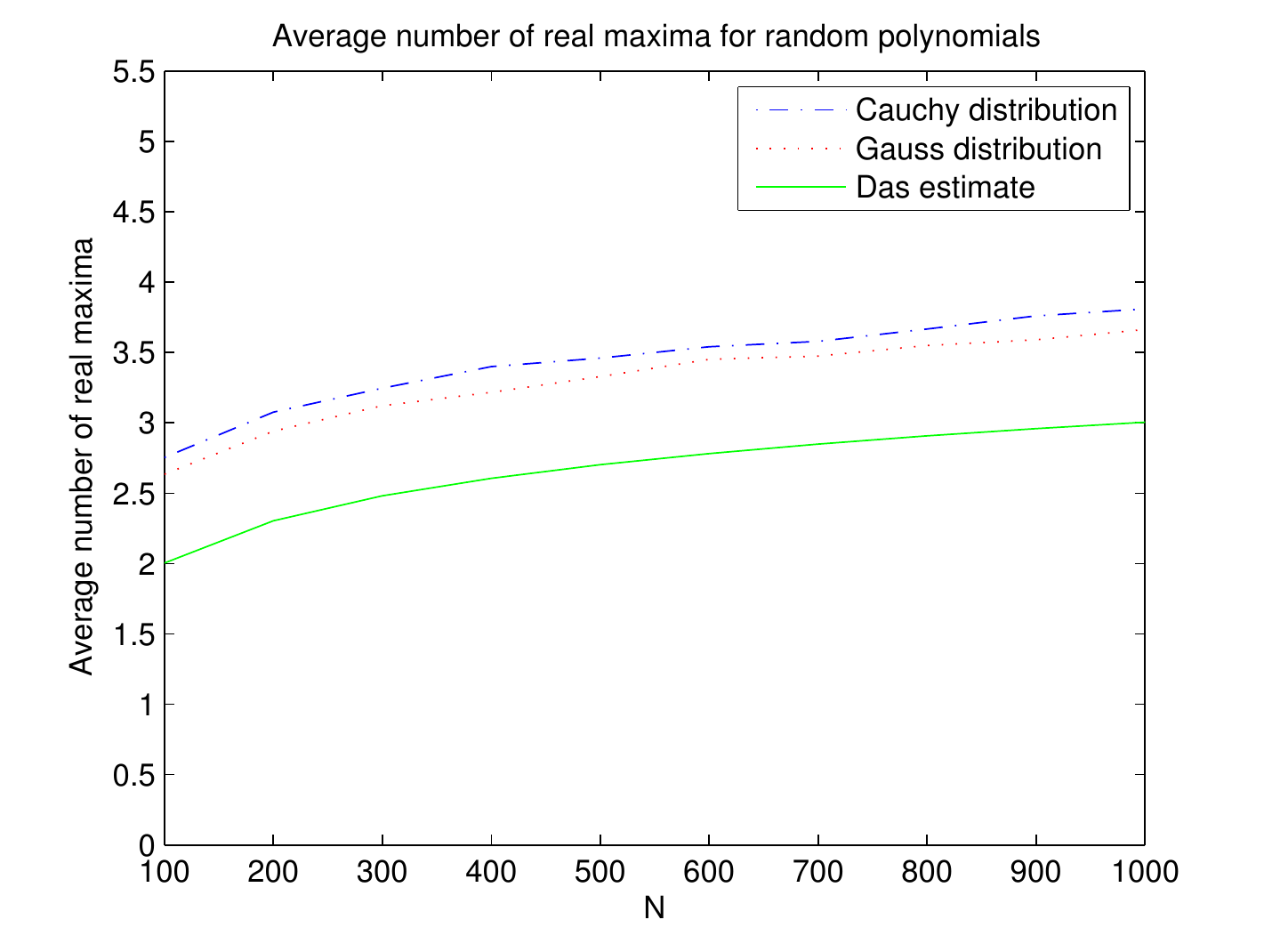}
    \caption{The mean number of real maxima for Cauchy and Gaussian random polynomials.
     The theoretical asymptotic that can be found in \cite{Farahmand:98} 
     is also reported.}
    \label{fig:cauchy_gauss_mean_maxima}
  \end{subfigure}
  \hspace{0.1in}
  \begin{subfigure}[t]{.45\textwidth} \centering
    \includegraphics[width=\linewidth]{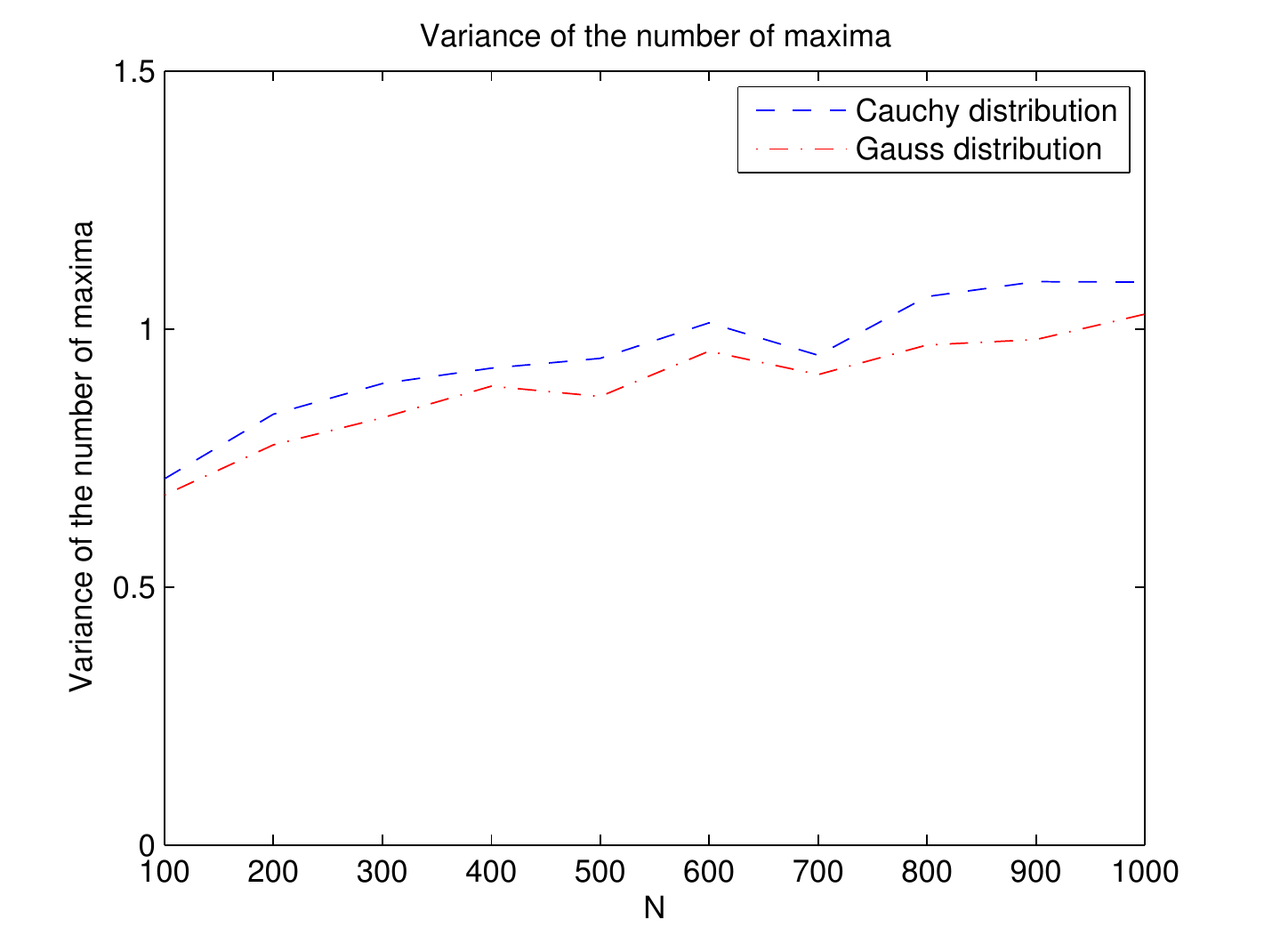}
    \caption{The variance of the number of real maxima for Gaussian and Cauchy
    polynomials.}
    \label{fig:maxima_variance}
  \end{subfigure}
  \caption{Experimental analysis of the distribution of the maxima.}
\end{figure}

In many areas of science and engineering, 
studying the surface defined by a 
multi-dimensional hypersurface described
by a multivariate nonlinear function, 
called potential energy function, 
is a fundamental problem -- called analyzing 
the potential energy landscape of the given system \cite{Wales:04}. 
In particular, a great deal of research is invested in finding stationary points of potential energy landscapes of different multivariate systems.
Finding stationary points of multivariate functions amounts to solving the system of equations obtained by all the first derivatives with 
respect to all the variables equated to zero. To obtain
statistical information about generic potential energy landscapes, 
it is customary to explore the potential energy landscapes of the random potentials, i.e., 
multivariate functions whose coefficients are randomly picked from a random distribution \cite{Greene:2013ida}.

Using our numerical setup, we mimic exploring the random potential energy landscapes of 
the univariate random potential $F(x)$, where
\begin{equation}
 F(x) = \sum_{i=0}^{D} a_{i} x^{i}.
\end{equation}
In other words, we solve $f(x) = F'(x) = 0$. The number of maxima can be obtained
by just dividing the number of real zeros of $F'(x)$ by $2$, since the number
of inflection points is of measure $0$ (see \cite{Farahmand:98}). 

\subsubsection{Average number of real critical points}

With our numerical setting we have performed a simulation on $2000$
random polynomials to check the theoretical estimates in \cite{Farahmand:98}. 
The results can be seen in Figure~\ref{fig:cauchy_gauss_mean_maxima}. 

\subsubsection{Variance of the number of maxima}
With the same data obtained above we can also estimate the variance of 
the number of zeros, which is summarized in Figure~\ref{fig:maxima_variance}.  We note that the analysis for the minima is totally symmetric. 

\subsubsection{Average number of minima at which $F(x)> 0$}
We now measure the minima for which $F(x)>0$. 
In a string theory set up, such minima are called {\em de~Sitter minima}. 
In this set up, $F(x)$ evaluated at a minimum corresponds to the cosmological constant. Since the observations have revealed that the 
cosmological constant is always positive definite, the search is then restricted to minima with $F(x)>0$.

It is easy to see that in our case if $F'(x) = 0$ and $F''(x) > 0$ then
by defining $G(x) := F(-x)$ we have that $G'(x) = 0$ and $G''(x) > 0$. In 
particular this implies that the mean number of minima points where $F(x) > 0$ are
bound to be equal to the mean number of minima points where $F(x)$ is negative due
to symmetry reasons. So we can expect that the number of positive minima will 
be half of the negative ones. 

The only question that we can ask is what can be said about the variance
of this number. 

The results of this analysis are reported in Figure~\ref{fig:variancepositivemaxima}. 
Similarly to the variance of positive roots,
we see that the variance is slowly increasing with the degree and it is 
also slowly converging with the size~of~the~sample. 

\begin{figure}[ht]
  \begin{center}
    \includegraphics[width=\linewidth]{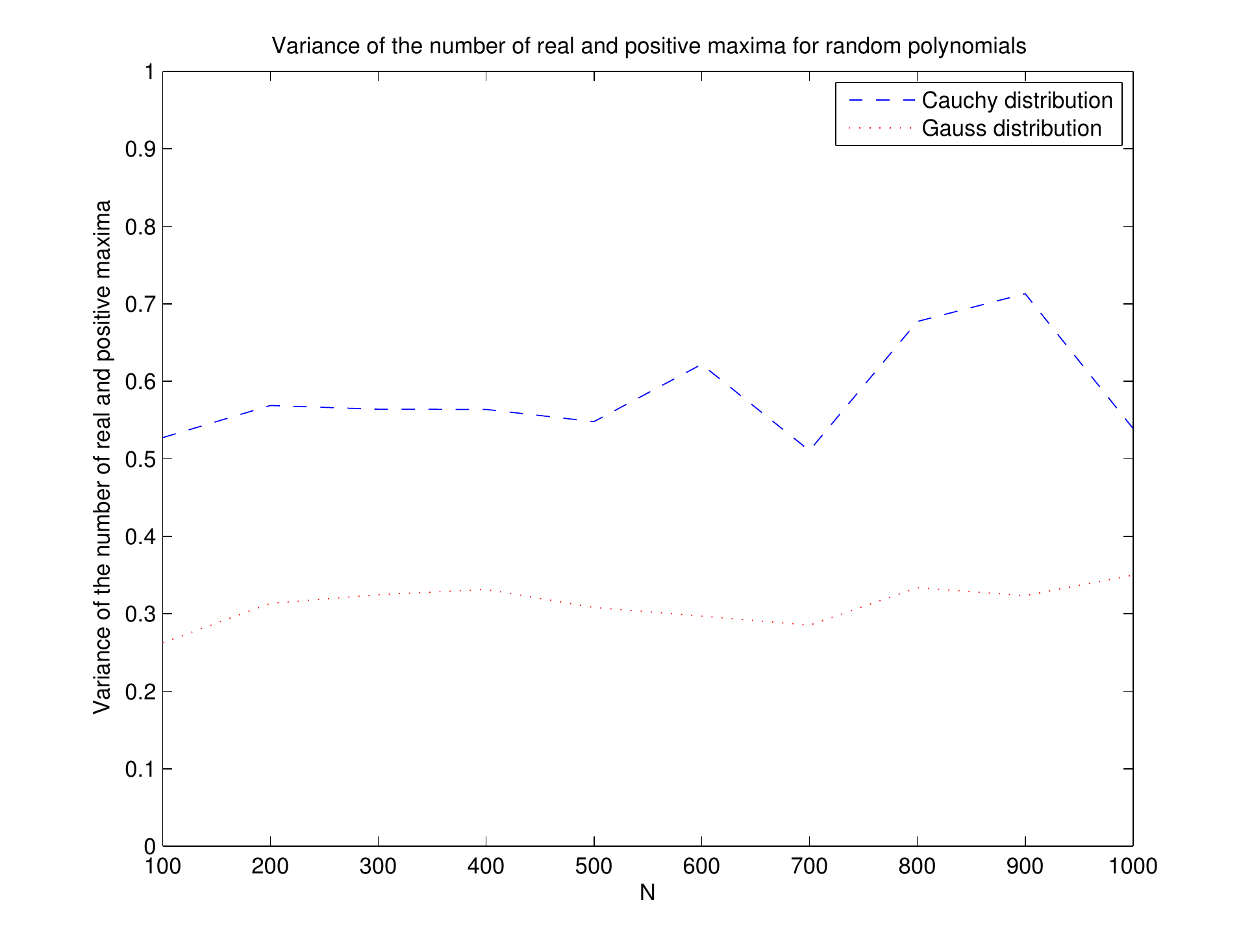}
  \end{center}
  \caption{Variance of the number of positive maxima of the random polynomials.}
  \label{fig:variancepositivemaxima}
\end{figure}

\subsubsection{The Histogram of $F''(x)$ at the Real Roots}
For multivariate random potentials, an important physical quantity is the eigenvalues of the Hessian matrix evaulated at the 
real stationary points. The distribution of the eigenvalues of random multivariate potentials is well studied in theoretical and statistical physics.
It is shown that in certain cases the distribution follows Wigner's semicircle law.
Here, for the univariate random potential case, we want to analyze the distribution of the evaluations of $F''(x)$ on the
real roots. 
Due to the very high oscillations of the value of $F''(x)$ given the high degree
of the polynomials, we were able to easily analyze only the case of degree $100$. 
We decided to numerically estimate the distribution of $\log \lvert F''(x) \rvert$
instead of estimating $F''(x)$ to obtain a much more meaningful plot. 

The expermiental results obtained for $N=100$ are reported in Figures~\ref{fig:cauchy_density_fpp} and \ref{fig:gauss_density_fpp}.

\begin{figure}[ht] \centering
  \begin{subfigure}[t]{.45\textwidth} \centering
    \includegraphics[width=\linewidth]{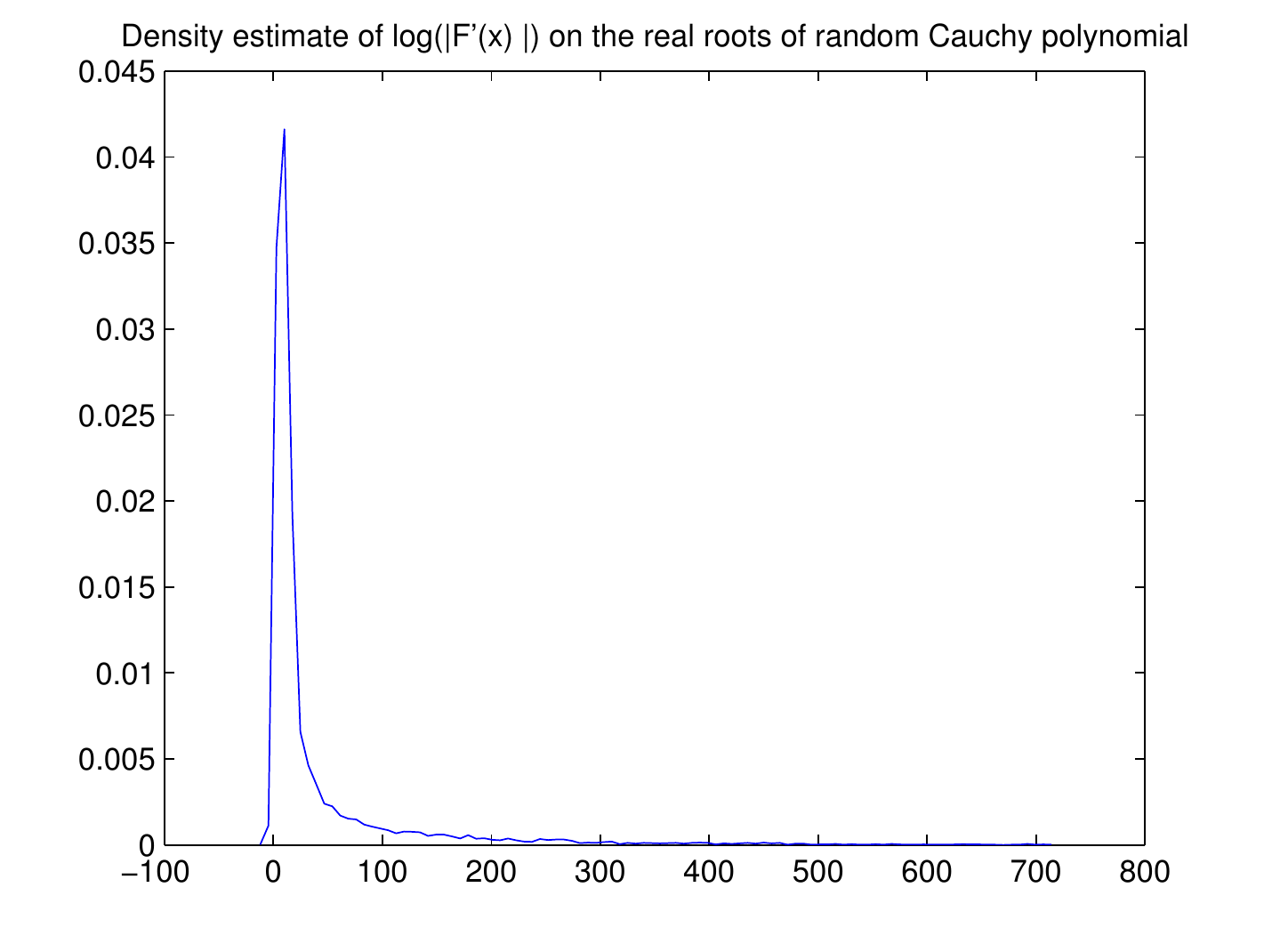}
    \caption{Cauchy polynomials.}
    \label{fig:cauchy_density_fpp}
  \end{subfigure}
  \begin{subfigure}[t]{.45\textwidth} \centering
    \includegraphics[width=\linewidth]{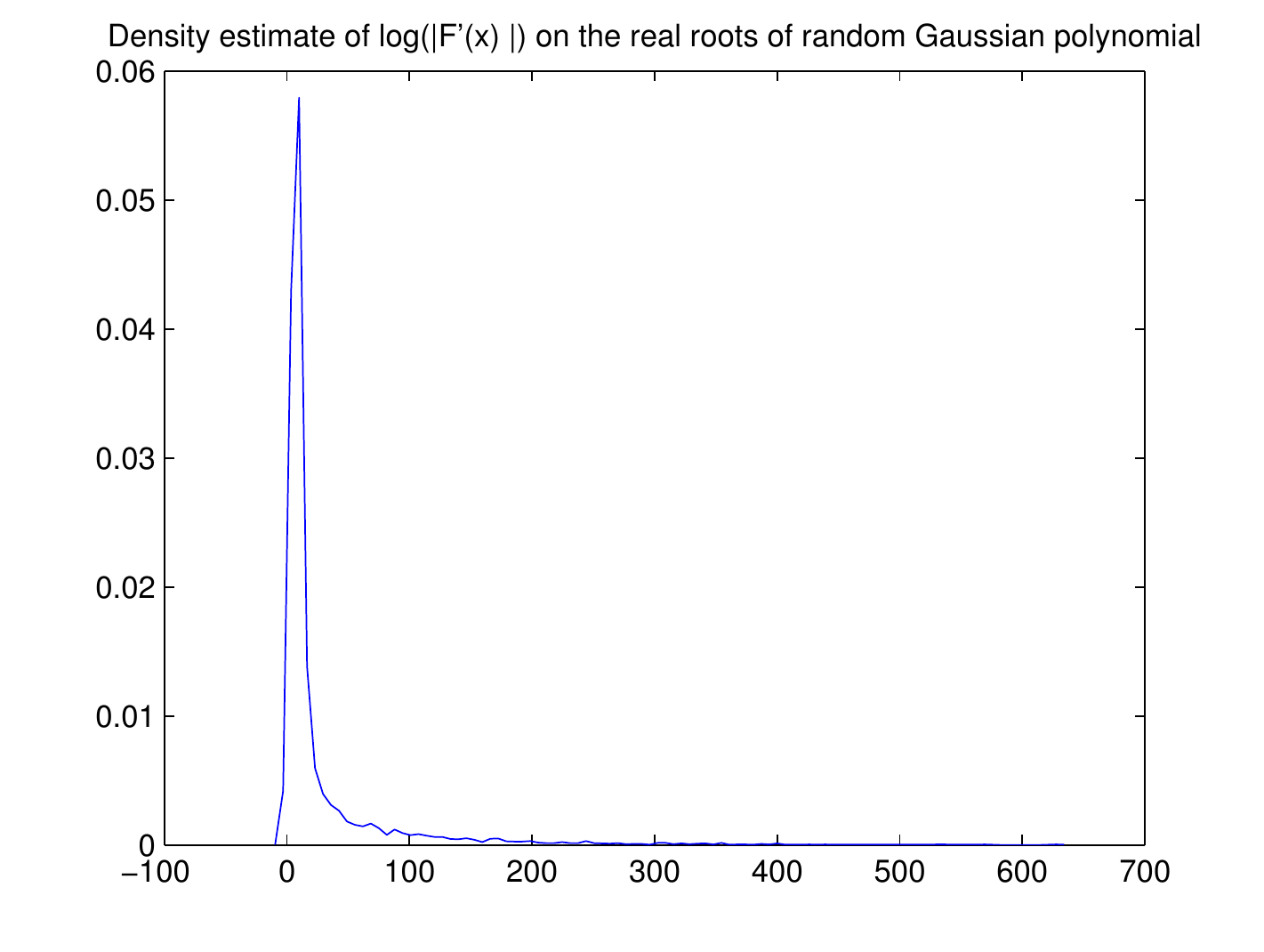}
    \caption{Gauss polynomials}
    \label{fig:gauss_density_fpp}
  \end{subfigure}
  \caption{Experimental analysis of the distribution of the evaulation of $\log \lvert F''(x) \rvert$ on the real roots of $F(x)$.}
\end{figure}


\section{Conclusion and Outlook} \label{sec:conclusions}
Finding roots of high-degree univariate polynomials with coefficients chosen from particular random distributions, 
though a very important task that is related to advances in many different areas
of science, mathematics and engineering, is a highly nontrivial task. Though there are several numerical root-finding methods available,
the numerical approximates found by these methods may turn out to be in the linear or even worse in the chaotic convergence region of the nearby
exact root. In this work, we attempted to find certified roots of such 
high-degree univariate polynomials.  In particular, we
used two different certification methods to certify if the numerical approximates are within the quadratic convergence region, 
namely based on Smale's $\alpha$-theory, a local method, 
and Gerschgorin's theorem, a global method.
We first find all the potential numerical approximates of a univariate polynomial using traditional methods 
such as the companion matrix method as implemented in
{\sc Matlab}'s {\tt roots()} command, or the multiprecision 
method implemented in MPSolve. Then, we certify if the numerical approximates are certifiable in the above mentioned sense, 
using both the certification methods via {\tt alphaCertified} and MPSolve.
We compared these methods based on the local conditioning 
constant $\gamma$ from $\alpha$-theory with $\gamma_M$ arising
from Gerschgorin's theorem.
In our experiment, the average of $\gamma$ increases 
exponentially with the degree whereas $\gamma_M$ increases polynomially
in the degree.

One advantage of Smale's $\alpha$-theory is that it is a local method.
That is, it can be performed using data computed at one point.
This allows for a quite independent and self-contained analysis
of the method.  On the other hand, relying 
on universal, worst-case methods often causes the bounds to 
grow very fast.  Global methods, which are based on
the minimum distance between the roots, can alleviate
the fast growth at the expense of computing all of the roots.

We emphasize that, given a high-degree polynomial,
multiprecision computations may be needed to perform the computations.
This is also true in cases where only floating 
point output ($16$ digits) is desired but, due to the ill-conditioning of the problem,  an intermediate algorithm implemented in multiprecision might be able to obtain all the $16$ digits. 

We also provided detailed experiments from the Cauchy case,
which is less studied than the Gaussian case. 
Our experiments suggest that:
\begin{enumerate}
  \item the available theoretical results are easily verifiable experimentally, as we have shown, for example, in Figure~\ref{fig:cauchy_gauss_mean} for the
    expected value for the number of real roots;
  \item there is some additional difficulties in the theoretical analysis
    of Cauchy random polynomials given by the fact that the Cauchy distribution does not have moments of any finite order, e.g., it is not possible to apply Maslova's theorem, but experimental results
    show that the variance of the number of real roots has a behaviour compatible with the Maslova's estimate.
\end{enumerate}
In conclusion, these results suggest a generalization to the Cauchy case 
is possible.

\end{document}